\documentclass{amsart}
\usepackage[utf8]{inputenc}

\usepackage{fullpage}

	\usepackage{amsmath}
 \usepackage{amsfonts}
	\usepackage{amssymb}
	\usepackage{amsthm}
	\usepackage{braket}
	\usepackage{xcolor}
\usepackage[shortlabels]{enumitem}
\usepackage{stmaryrd}

 \usepackage{float}
	\usepackage{graphicx}
	\usepackage{hyperref}
    \usepackage{cleveref}
	\usepackage{ifsym}
	\usepackage{indentfirst}
	\usepackage{mathrsfs}
	\usepackage{mathtools}
	\usepackage{proof}
	\usepackage{qtree}
	\usepackage{setspace}
	\usepackage{tensor}
 \usepackage{tabularx}
	\usepackage{tikz}
	\usepackage{tikz-cd}
	\usepackage{tabstackengine}
        \usepackage{cancel}

\usepackage{nameref}

\usepackage{scalerel}

\makeatletter
\newsavebox{\@abr}
\newcommand{\llangle}[1][]{\savebox{\@abr}{\(\m@th{#1\langle}\)}%
  \mathopen{\copy\@abr\mkern2mu\kern-0.9\wd\@abr\usebox{\@abr}}}
\newcommand{\rrangle}[1][]{\savebox{\@abr}{\(\m@th{#1\rangle}\)}%
  \mathclose{\copy\@abr\mkern2mu\kern-0.9\wd\@abr\usebox{\@abr}}}

  \newsavebox{\@sbr}
\newcommand{\llsquare}[1][]{\savebox{\@sbr}{\(\m@th{#1{[}}\)}%
  \mathopen{\copy\@sbr\mkern3mu\kern-0.9\wd\@sbr\usebox{\@sbr}}}
\newcommand{\rrsquare}[1][]{\savebox{\@sbr}{\(\m@th{#1{]}}\)}%
  \mathclose{\copy\@sbr\mkern3mu\kern-0.9\wd\@sbr\usebox{\@sbr}}}
\makeatother

\SetSymbolFont{stmry}{bold}{U}{stmry}{m}{n}

	\providecommand{\corollaryname}{Corollary}
	\providecommand{\definitionname}{Definition}
	\providecommand{\examplename}{Example}
	\providecommand{\lemmaname}{Lemma}
	\providecommand{\propositionname}{Proposition}
	\providecommand{\remarkname}{Remark}
	\providecommand{\theoremname}{Theorem}
	\providecommand{\setupname}{Setup}
	\providecommand{\conjecturename}{Conjecture}
	\providecommand{\questionname}{Question}
	\providecommand{\objectivename}{Objective}
	\providecommand{\aimname}{Aim}
	\providecommand{\notationname}{Notation}

	\theoremstyle{plain}
		\newtheorem{thm}{\protect\theoremname}[section] 
		\newtheorem{prop}[thm]{\protect\propositionname}
		\newtheorem{lem}[thm]{\protect\lemmaname}

	\theoremstyle{definition}
		\newtheorem{defn}[thm]{\protect\definitionname}
		\newtheorem{example}[thm]{\protect\examplename}
		\newtheorem{setup}[thm]{\protect\setupname}

		\newtheorem{notn}[thm]{\protect\notationname}

	\theoremstyle{remark}
		\newtheorem{rem}[thm]{\protect\remarkname}
		
	\numberwithin{figure}{section}
	\numberwithin{equation}{section}

	\usetikzlibrary{matrix,arrows,decorations.pathmorphing,positioning,decorations.pathreplacing}
	\tikzset{commutative diagrams/.cd, 
		mysymbol/.style = {start anchor=center, end anchor = center, draw = none}}
    \tikzset{
    labl/.style={anchor=north, rotate=90, inner sep=2mm}
    }

	\tikzcdset{every label/.append style = {font = \footnotesize}}
	\tikzcdset{scale cd/.style={every label/.append style={scale=#1},
    cells={nodes={scale=#1}}}}

		\newcommand{\rad}{\operatorname{rad}\nolimits}

		\newcommand{\cok}{\operatorname{coker}\nolimits}
		\newcommand{\im}{\operatorname{im}\nolimits}
		\newcommand{\coim}{\operatorname{coim}\nolimits}

\newcommand{\lMod}[1]{{#1}\operatorname{\hspace{-0.4mm}-\hspace{0.2mm}\mathsf{Mod}}\nolimits}

\newcommand{\lRel}[1]{{#1}\operatorname{\hspace{-0.4mm}-\hspace{0.2mm}\mathsf{Rel}}\nolimits}
\newcommand{\lRep}[1]{{#1}\operatorname{\hspace{-0.4mm}-\hspace{0.2mm}\mathsf{Rep}}\nolimits}

  \newcommand{\Hom}[1]{\tensor*[]{\operatorname{Hom}}{_{#1}}}
  \newcommand{\End}[1]{\tensor*[]{\operatorname{End}}{_{#1}}}

 	\newcommand{\jacob}{J}

    \newcommand{\Kron}{\text{\scalebox{1.4}{$\varkappa$}}}
	
    \newcommand{\nab}[1]{\nabla_{\hspace{-0.75mm}#1}}
 \newcommand{\CA}{\mathcal{A}}
\newcommand{\CB}{\mathcal{B}}
\newcommand{\CC}{\mathcal{C}}

\newcommand{\CF}{\mathcal{F}}
\newcommand{\CG}{\mathcal{G}}

\newcommand{\CS}{\mathcal{S}}
\newcommand{\CT}{\mathcal{T}}

\newcommand{\graph}[1]{\operatorname{graph}(#1)}

\newenvironment{acknowledgements}{
		\begin{abstract}} {\end{abstract}}

\newcommand*\circled[1]{\tikz[baseline=(char.base)]{
  \node[shape=circle,draw,inner sep=1pt] (char) {
  \scaleobj{0.75}{#1}
  }}}
  \newcommand{\Rceil}{\scalebox{0.8}{$\Omega$}}
  \newcommand{\Rfloor}{\scalebox{0.8}{$\boldsymbol{\mho}$}}
  
\newcommand{\idfunc}[1]{\boldsymbol{1}_{#1}}
\newcommand{\tens}[1]{{#1}^{\circled{\phantom{\hspace{1mm}}}}}
\newcommand{\homs}[1]{{#1}^{_{\square}}}

\newcommand{\exact}[3]{\big({}_{\mathcal{#2}}\hspace{-0.8mm}\overset{{}_{{#1}}}{*}_{ \mathcal{#3}}\big)}

\makeatother

\begin{document}
\title{Linear relations over commutative rings}
\author{Raphael Bennett-Tennenhaus}
\address{Fakult\"at f\"ur Mathematik, Universit\"at Bielefeld, 33501 Bielefeld, Germany}
\email{raphaelbennetttennenhaus@gmail.com}



\thanks{}

\begin{abstract}
We consider the category of  linear relations over an arbitrary commutative ring, and identify it as a subcategory of the category of Kronecker representations.  
We observe that this subcategory forms a definable, faithful and hereditary torsion-free class. 
We also generalise  results used in the functorial filtrations method, known before only in case the ground ring is a field. 
In particular, our results strictly generalise what the so-called the `covering' and `splitting' properties from this method.  
\end{abstract}
\maketitle

\section{Introduction}

In representation theory the notions of finite, tame and wild representation type play a central role. 
The representation type of a finite-dimensional algebra over an algebraically closed field is determined by how one can parametrize the isoclasses of indecomposable finite-dimensional modules.
Said type is: finite if there are finitely many isoclasses;  tame if, for each fixed dimension, the isoclasses can be described by a finite number of one-parameter families; and wild if the problem of classifying these isoclasses contains the problem of finding a normal form of pairs of square matrices under simultaneous conjugation by a non-singular matrix. 

A theorem of Drozd says that these types are mutually exclusive, and that the type of any such algebra must be finite, tame, or wild. 
A theorem of Gabriel says that the type of a path algebra over a field is finite or tame if and only if the underlying graph of the quiver is classical or extended Dynkin $\mathbb{A}\text{-}\mathbb{D}\text{-}\mathbb{E}$. 
A prototypical example of a tame algebra is the path algebra $K\Kron$ over a field $K$ where $\Kron=\begin{tikzcd}
    t
    \arrow[r, shift left=0.75]
    \arrow[r, shift right=0.75]
    &
    h.
\end{tikzcd}$

In what follows we study more generally the path algebra $R\Kron$ where $R$ is a commutative ring. 
Modules over $R\Kron$ have a close connection to the concept of a \emph{linear relation}. 
For left $R$-modules $L$ and $M$, an \emph{$R$-linear relation from}  $L$ \emph{to}  $M$ is a submodule $C$ of $L\oplus M$. 
A prototypical example of an $R$-linear relation  is $\mathrm{graph}(\theta)\coloneqq\{(\ell,\theta(\ell))\mid \ell\in L\}$ for some $\theta\in\Hom{R}(L,M)$. 
An $R$-linear relation from $M$ to $M$ is referred to as an $R$-linear relation \emph{on} $M$. 
Objects in the category $\lRel{R}$ are  pairs $(M,C)$ where $C$ is an $R$-linear relation on $M$. 
Morphisms $(L,B)\to (M,C)$ are defined by $R$-linear maps $f\colon L\to M$ such that $(f(x),f(y))\in C$ whenever $(x,y)\in B$. 
In  \Cref{thm-main-general-intro} we study homological properties of this category. 
\begin{thm}
\label{thm-main-general-intro}   
$\lRel{R}$ is equivalent to a subcategory of $\lMod{R\Kron}$ that is precovering and closed under extensions,  limits, filtered colimits and coproducts and contains the regular representation.
Consequently, $\lRel{R}$ is quasi-abelian and  a covering, enveloping, definable, faithful and hereditary torsion-free class in $\lMod{R\Kron}$. 
\end{thm}
Considering linear relations as modules over the path algebra of $\Kron$ is not new, and served as a useful perspective in joint work with Crawley-Boevey \cite{BenCra2018} during an application of the  so-called \emph{functorial filtrations} classification method. 
This method has always relied on a careful analysis of linear relations defined by generators of the algebra, considered as subspaces of the underlying vector space of an arbitrary module. 
Especially useful in this analysis is to obtain a \emph{covering} statement and a \emph{splitting} statement; see for example \cite[Lemma~3.1]{BenCra2018} and \cite[Corollary~1.3]{BenCra2018} respectively. 
In the authors' thesis \cite{Ben2018} the functorial filtrations method was adapted to situations where the ground field is replaced by a complete noetherian local ring.


\emph{Splitting} concerns sequences of consecutive pairs that all lie in a common relation, and states that the subset of sequences that are eventually trivial has a complement, and that this complement satisfies certain closure operations. 
In the adaptation of the functorial filtrations method in \cite{Ben2018}, one observes that this splitting statement is both unnecessary for the purposes of classification, as well as being false when the local ring is not a field. 
However, the splitting result can be generalised suitably for classification. 

For an $R$-linear relation $C$ from $L$ to $M$ the \emph{inverse} is $C^{-1}= \{(m,\ell)\in M\oplus L\colon(\ell,m)\in C\}$.
For any $\ell\in L$ we let $C\ell=\{m\in M\colon (\ell,m)\in C\}$. 
Assuming $L=M$, in this notation define $R$-submodules of $M$ by
\begin{equation}
    \label{eqn-splitting-thm-submodules}
     \begin{array}{c}
C'' = \{ m\in M \colon  \text{$\exists \,(m_{n})\in M^{\mathbb{N}}$ with $m_{n+1}\in Cm_{n}$ for all $n$ and $m=m_{0}$}\},
\\
C' = \{ m\in M \colon  \text{$\exists \,(m_{n})\in M^{\mathbb{N}}$ with $m_{n+1}\in Cm_{n}$ for all $n$, $m=m_{0}$ and  $m_{n} = 0$ for $ n\gg0$}\},
    \\
       \begin{array}{cc}
 C^{\sharp} = C'' \cap (C^{-1})'',
 &
 C^{\flat} = C'' \cap (C^{-1})'+(C^{-1})'' \cap C'.
        \end{array}
        \end{array}
\end{equation}
 The aforementioned generalisation of the splitting result comprises \Cref{thm-intro-2}, which generalises  \cite[Lemmas~4.4,~4.5~and~4.6]{Cra2018}; see \Cref{lem-sharp-over-flat-functorial}, \Cref{lem-reductions-define-surjective-automorphism-respecting-maps}, \Cref{prop-reductions-generalise-relations-that-split},  \Cref{thm-existence-of-reductions} and \Cref{example-thesis-relation-ii}. 
\begin{thm}
\label{thm-intro-2} The assignment $(C,M)\mapsto C^{\sharp}/C^{\flat}$  gives a functor $\sharp/\flat\colon \lRel{R}\to \lMod{R[T,T^{-1}]}$  with a right inverse. 
Furthermore, if $R$ is local with jacobson radical $\jacob$, if $ \jacob M\subseteq (C^{-1})'+C'$ and if $C^{\sharp}/C^{\flat}$ is finitely generated, then there exists an $R[T,T^{-1}]$-module $X$,  free over $R$, and  an $R$-linear map $\rho\colon X\to M$ such that 
\[
\begin{array}{ccc}
     C^{\sharp}=C^{\flat}+\im(\rho),
     &
     \rho^{-1}(C^{\flat})=JX,
     &
     \rho(Tx)\in C\rho(x)\text{ for all }x\in X.
\end{array}
\]
Thus $\rho$ defines a morphism in $\lRel{R}$ which under $\sharp/\flat$ gives an exact sequence $0\to JX\to X\to C^{\sharp}/C^{\flat}\to 0$. 
\end{thm}
\emph{Covering} is set in the context of a diagram in $\lMod{R}$ whose underlying graph is the binary tree. 
For a non-zero element in the root, the statement roughly says that one can find a non-trivial sequence of elements in a ray of the binary tree that cover this element in a minimal way. 
To be precise, let $\mathcal{B}$ be the quiver
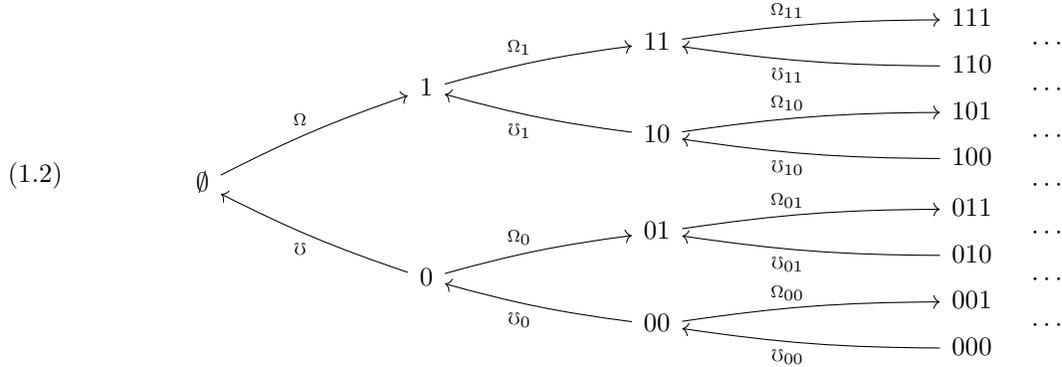
\begin{equation}
\label{binary-tree}
    \begin{tikzcd}[column sep = 2.5cm, row sep = -5pt]
 &  &  & 111
\\
 &  & 11\arrow[ur, bend left = 4, "{\Rceil_{11}}"] &  \hspace{2cm}\cdots
\\
 &  &  & 110\arrow[ul, bend left = 4, "{\Rfloor_{11}}"]
\\
 & 1\arrow[uur, bend left = 4, "{\Rceil_{1}}"] & &  \hspace{2cm}\cdots
\\
 &  &  & 101
\\
 &  & 10\arrow[uul, bend left = 4, "{\Rfloor_{1}}"]\arrow[ur, bend left = 4, "{\Rceil_{10}}"] &  \hspace{2cm}\cdots
\\
 &  &  & 100\arrow[ul, bend left = 4, "{\Rfloor_{10}}"]
\\
\emptyset\arrow[uuuur, bend left = 4, "{\Rceil}"] &  &  &  \hspace{2cm}\cdots
\\
 &  &  & 011
\\
 &  & 01\arrow[ur, bend left = 4, "{\Rceil_{01}}"] & \hspace{2cm}\cdots
\\
 &  &  & 010\arrow[ul, bend left = 4, "{\Rfloor_{01}}"]
\\
 & 0\arrow[uuuul, bend left = 4, "{\Rfloor}"]\arrow[uur, bend left = 4, "{\Rceil_{0}}"] &  & \hspace{2cm}\cdots
\\
 &  &  & 001
\\
 &  & 00\arrow[uul, bend left = 4, "{\Rfloor_{0}}"]\arrow[ur, bend left = 4, "{\Rceil_{00}}"] & \hspace{2cm}\cdots
\\
 &  &  & 000\arrow[ul, bend left = 4, "{\Rfloor_{00}}"]
\end{tikzcd}
\end{equation}
By a \emph{ray} in $\mathcal{B}$ we mean a sequence $\overline{\sigma}=(\sigma^{0},\sigma^{1},\sigma^{2},\dots)$ of depth-$i$ vertices  $\sigma^{i}$, necessarily starting with $\sigma^{0}=\emptyset$, such that $\sigma^{i+1}$ is the extension of $\sigma^{i}$ to the right by $0$ or $1$. 
For example $(\emptyset,0,01,011,0110,01101,011010,\dots)$. 
Given such a ray and a $\mathcal{B}$-diagram $\mathscr{M}$ in $\lMod{R}$ we shorthand notation by writing $M[i]$ for the image of $\sigma^{i}$ under $\mathscr{M}$ and $\mu_{i}$ for the image of the unique arrow between  $\sigma^{i}$ and $\sigma^{i+1}$ under $\mathscr{M}$. 
For example 
\[
\begin{tikzcd}[column sep = 1.2cm]
M[0]
&
M[1]\arrow[l, "{\mu_{1}}"', "{\mathscr{M}(\Rfloor)}"]\arrow[r, "{\mu_{2}}", "{\mathscr{M}(\Rceil_{0})}"']
&
M[2]
\arrow[r, "{\mu_{3}}", "{\mathscr{M}(\Rfloor_{01})}"']
&
M[3]
&
M[4]\arrow[l, "{\mathscr{M}(\Rfloor_{011})}", "{\mu_{4}}"']
\arrow[r, "{\mu_{5}}", "{\mathscr{M}(\Rceil_{0110})}"']
&
M[5]
&
M[6]\arrow[l, "{\mathscr{M}(\Rfloor_{01101})}", "{\mu_{6}}"']
\arrow[r]
&
\cdots
\end{tikzcd}
\]
Writing $\widetilde{\mu}_{i}$ for the relation from $M[i-1]$ to $M[i]$ defined either by $\mathrm{graph}(M(\Rceil_{\sigma^{i}}))^{-1}$ or $\mathrm{graph}(M(\Rfloor_{\sigma^{i}}))$, one can generate submodules of $\mathscr{M}(\emptyset)=M[0]$ by  composing these relations. 
For example if $L\subseteq M[5]$ then 
\[
\widetilde{\mu}_{1}\widetilde{\mu}_{2}\widetilde{\mu}_{3}\widetilde{\mu}_{4}\widetilde{\mu}_{5}\widetilde{\mu}_{6}L=\left\{m_{0}\in M[0]\left\vert\, \begin{array}{c}
\text{there exists } m_{i}\in M[i]\text{ for }i=1,\dots,5\text{ such that }m_{5}\in L\text{ and}
\\
m_{0}=\mu_{1}(m_{1}),\,
\mu_{2}(m_{1})=m_{2},\,
\mu_{3}(m_{2})=m_{3}=\mu_{4}(m_{4}),\,
\mu_{5}(m_{4})=m_{5}
\end{array}\right.\right\}.
\]
In general we define $\widetilde{\mu}_{\leq 0}\coloneqq\graph{\idfunc{M[0]}}$ and, for each $n> 0$, we define $\widetilde{\mu}_{\leq n}\coloneqq \widetilde{\mu}_{1}\dots \widetilde{\mu}_{n}$. 

\Cref{thm-main-3-intro} below is essentially \Cref{thm-main-3} and is inspired by \cite[Lemmas~10.2~and~10.3]{Cra2018}.
%
%
\begin{thm}
\label{thm-main-3-intro}
   Let $R$ be noetherian, semilocal and complete in the adic topology given by its jacobson radical.  
    Let $\mathscr{M}$ be a $\mathcal{B}$-diagram in $\lMod{R}$ such that $\mathscr{M}(\sigma)$ is finitely generated and  $\im (\mathscr{M}(\Rfloor_{\sigma}))\subseteq \ker(\mathscr{M}(\Rceil_{\sigma}))$ for all $\sigma$. 
    If $0\neq m\in M(\emptyset)$ then there exists a ray $\overline{\sigma}$ in $\mathcal{B}$
such that exactly one of the following statements holds.
\begin{enumerate}
\item There exists $j\in \mathbb{N}$ such that $ m\in \widetilde{\mu}_{\leq j} \ker(M(\Rceil_{\sigma^{j}}))\setminus \widetilde{\mu}_{\leq j}  \im (M(\Rfloor_{\sigma^{j}}))$. 
\item There exists $(m_{i})\in \prod _{i\in \mathbb{N}}M[i]$ such that $m_{0}=m$ and $m_{i-1}\in \widetilde{\mu}_{i}m_{i}$ when $i>0$, but $m\notin \bigcup_{i\in \mathbb{N}}\widetilde{\mu}_{\leq i}0$. 
\end{enumerate}
\end{thm}


\section{Preliminaries}
\label{sec-prelims}

    All subcategories we consider in what follows will be assumed to be full, unless stated otherwise.  
In \S\ref{sec-prelims} we recall well-known definitions and results on covering theory and torsion theory for quasi-abelian categories, and then relate it to purity and definability in the context of a module category, finally noting the structure theory for modules over a ring of Morita context. 
We gradually restrict the setting we work in.

\subsection{Covering theory }
\label{subsubsec-covering}

We begin by recalling generalisations of projective covers and injective envelopes.  
We follow work of Enochs \cite{enochs-covers} and later we follow work of Rada and Saorin \cite{rada-saorin-rings-characterised}. 
Throughout \S\ref{sec-prelims} we let $\CA$ be a preadditive category. 
Later we add more restrictions to $\CA$.

\begin{defn}
        \label{defn-functorial-finiteness} 
\cite[\S1,~p.~190]{enochs-covers}
(c.f.~\cite[\S1,~p.~900]{rada-saorin-rings-characterised}) 
Let $\CC$ be a subcategory of $\CA$ and $b\in \Hom{\CA}(M,N)$. 

We say $b$ is \emph{left}  (respectively, \emph{right}) \emph{minimal} if every $a\in\End{\mathcal{\CA}}(N)$ with $ab=b$ (respectively, $c\in\End{\mathcal{\CA}}(M)$ with $bc=b$) is an isomorphism.  
We say $b$ is a $\CC$-\emph{preenvelope of} $M$ if $N\in \CC$ and for any  $c\in \Hom{\CA}(M,C)$ with $C\in \CC$ we have $c=ab$ for some $a\in  \Hom{\CC}(N,C)$. 
The notion of a  $\CC$-\emph{precover} is dual. 
A $\CC$-\emph{envelope} (respectively, $\CC$-\emph{cover}) is a left (respectively, right) minimal $\CC$-preenvelope (respectively, $\CC$-precover). 
We say that $\CC$ is  \emph{enveloping} (respectively, \emph{covering})  if every object in $\CA$ admits a $\CC$-envelope (respectively, $\CC$-cover).         
    \end{defn}

\subsubsection{}
\label{para-AS-alt-terms} 
\cite[p.~81]{auslander-smalo-preprojectives-over-artin-algs} 
Auslander and Smal{\o} used  alternative terminology to refer to these notions. 
 When  $\CC$  is closed under direct sums and summands one says  $\CC$ is \emph{covariantly} (respectively, \emph{contravariantly}) {finite} in $\CA$ provided every object in $\CA$ admits a $\CC$-preenvelope (respectively, $\CC$-precover). 
    To say that $\CC$ is  \emph{functorially finite} in $\CA$ means  $\CC$ is both covariantly finite and contravariantly finite in $\CA$. 
 It is straightforward to see that epimorphic preenvelopes are envelopes and monomorphic precovers are covers. 

\subsubsection{}
\label{para-co-complete-covering-torsion} 
\cite[Propositions~4.1,~4.11]{rada-saorin-rings-characterised} 
If  $\CA$ is a category of modules over a ring, then  a  subcategory $\CC$ of  
 $\CA$ is closed under subobjects (meaning the domain of any morphism, whose codomain lies in $\CC$, must also lie in $\CC$)  and products (meaning the product of any set of objects, each of which lies in $\CC$, must also lie in $\CC$) if and only if any object in $\CA$ has a surjective $\CC$-envelope. 
 Dually $\CC$ is closed under quotients and coproducts if and only if any object in $\CA$ has a monomorphic $\CC$-cover.

 \subsection{Exactness}
 \label{subsec-exactness}
 From now on assume $\CA$ is additive. 
 We recall what it means for $\CA$ to be \emph{exact} as introduced by Quillen \cite{quillen-higher-algebraic-K}; see also work of B{\"u}hler \cite{buhler-exact}. 
   Consider a commutative square in $\CA$ of the form
     \[
\begin{tikzcd}
 L\arrow[d, "w"']\arrow[r, "x"] & M\arrow[d, "z"]
\\
 N\arrow[r, "y"'] & P
\end{tikzcd}
\]
When this square is \emph{cartesian}, meaning the pullback of $y$ and $z$, we refer to $x$ as \emph{the pullback of $y$ along $z$} and $w$ as \emph{the pullback of $z$ along $y$}.    
There are dual notions for when the square is \emph{cocartesian}. 
One says that a class $\mathcal{M}$ of morphisms in $\CA$ is \emph{stable under pullback} if the pullback of any morphism in $\mathcal{M}$ lies in $\mathcal{M}$. 
Dually one can define when $\mathcal{M}$ is \emph{stable under pushout}. 
We follow B{\"u}hler \cite[Definition~2.1]{buhler-exact}.

\begin{defn}
\cite[\S2]{quillen-higher-algebraic-K}   
    such that $x$ is the kernel of $y$ and such that  $y$ is the cokernel of $x$. 
    An \emph{admissible monic}  (respectively, \emph{epic}) refers to a morphism $x$ (respectively, $y$) such that there exists a morphism  $y$ (respectively, $x$) for which $(x,y)$ is a kernel-cokernel pair. 

    By an \emph{exact structure} on $\CA$ one means a class $\mathcal{E}$ of kernel-cokernel pairs such that:  
    the identity morphism on any object is simultaneously an admissible monic and an admissible epic; 
    the pushout (respectively, pullback) of an admissible monic (respecitvely, epic) along any morphism exists; 
    and the class of admissible monics (respectively, epics) are closed under composition and stable under pushout (respectively, pullback). 
\end{defn}

\subsection{Quasi-abelian categories}
\label{subsec-quasi-abel}
From now on assume the additive category $\CA$ is \emph{preabelian}, meaning every morphism has a kernel and a cokernel; see the book of Bucur and Deleanu \cite{Bucur-Deleanu-intro}.
Next we recall a generalisation of abelian categories due to Schneiders \cite{Schneiders-quasi-abelian}. 
Rump \cite{rump-almost-abelian} used the terminology of \emph{almost abelian} categories. 

\subsubsection{}
\label{para-preable} 
\cite[Definition~1.1.1,~Remark~1.1.2]{Schneiders-quasi-abelian} 
Any preabelian category such as $\CA$ has all pullbacks and pushouts, and every morphism $x\colon A\to B$ has a \emph{coimage} $c_{x}\colon A\to \coim(x)$, the cokernel of its kernel,  as well as an \emph{image} $i_{x}\colon \im(x)\to B$, the kernel of its cokernel. 
    It follows that any such $x$ factors as $x=i_{x}\overline{x}c_{x}$ where $\overline{x}\colon \coim(x)\to \im(x)$ is the \emph{parallel} morphism. 
    Note that if $x$ is a monomorphism then $c_{x}$ is an isomorphism and hence $\overline{x}$ is a monomorphism. 
    One says $x$ is  \emph{strict} provided $\overline{x}$ is an isomorphism. 
    A kernel of a morphism is always strict. 
    Conversely, as discussed above, if $x$ is a strict monomorphism then it factors through $i_{x}$ by an isomorphism. 
    Dually $x$ is a strict epimorphism if and only if it factors through $c_{x}$ by an isomorphism. 

\begin{defn} 
\cite[Definition~1.1.3]{Schneiders-quasi-abelian} 
(see also \cite[\S1,~p.~167]{rump-almost-abelian}). 
We say the preabelian category $\CA$ is \emph{quasi}-\emph{abelian} if cokernels in $\CA$ are stable under pullback and kernels in $\CA$ are stable under pushout. 
\end{defn}

As discussed above  the class of strict monomorphisms (respectively, epimorphisms) coincides with the class of kernels (respectively, cokernels) and so the notions from  \cite{rump-almost-abelian}  and \cite{Schneiders-quasi-abelian} coincide. 
We now recall the way in which quasi-ablelian categories are exact, and hence why we can refer to short exact sequences. 

\subsubsection{}
\label{para-exactness-mono}     \cite[Corollary~1.5, Remarks~1.1.10,~1.1.11]{Schneiders-quasi-abelian}  
    Adopt the notation from \ref{para-preable}. 
    Given morphisms $x\colon A\to B$ and $y\colon B\to C$ such that $yx=0$, then since every kernel is a monomorphism and every cokernel is an epimorphism, we also have $\overline{y} c_{y}i_{x}\overline{x}=0$. 
    If also $\overline{y}$ is a monomorphism and $\overline{x}$ is an epimorphism, then $c_{y}i_{x}=0$ and so $i_{x}$ factors through a morphism $m_{xy}\colon \im(x)\to \ker(y)$, necessarily a monomorphism. 

\subsubsection{}
\label{para-strictly-exact}    \cite[Proposition~1.1.7]{Schneiders-quasi-abelian} 
Adopt the notation from \ref{para-preable} and \ref{para-exactness-mono}. 
If $\CA$ is quasi-abelian  the sequence $A\to B\to C$ is \emph{strictly exact} (respectively \emph{coexact}) at $Y$ provided $m_{xy}$ is an isomorphism and provided $x$ (respectively, $y$) is strict. 
The sequence  $0\to A\to B\to C\to 0$ is strictly exact at $A$, $B$ and $C$ if and only if it is strictly coexact at $A$, $B$ and $C$ if and only if $(x,y)$ is a kernel-cokernel pair. 
Thus any quasi-abelian category $\CA$ equipped with the class of strictly (co)exact sequences is exact as in \ref{subsec-exactness}.

\subsection{Orthogonal subcategories} 
\label{subsect-orthogonal}
From now on assume $\CA$ is quasi-abelian. 
We recall language required for the notion of \emph{torsion pairs}, recalled later. 
Let $\CT$ and $\CF$ be subcategories of $\CA$. 
Define the subcategories
    \[
    \begin{array}{cc}
    \CT^{\perp}\coloneqq \{F\in\CA\mid \Hom{\CA}(T,F)=0\text{ for all }T\in \CT\},
    &
    {}^{\perp}\CF\coloneqq \{T\in\CA\mid \Hom{\CA}(T,F)=0\text{ for all }F\in \CF\}.
    \end{array}
    \]
    Let $\CT*\CF$ be the subcategory of $M\in\CA$ such that there exists ${}_{\CT}M\in\CT$, $M_{\CF}\in\CF$ and an exact sequence 
    \[
    \begin{array}{cc}
    \begin{tikzcd}
        0\arrow[r]
        &
        {}_{\CT}M\arrow[r, "i^{\CT}_{M}"]
        &
        M\arrow[r,   "p_{M}^{\CF}"]
        &
        M_{\CF}\arrow[r]
        &
        0
    \end{tikzcd}
    &
    \exact{M}{T}{F}
    \end{array}
    \]
    in $\CA$, which we recall means that $(i^{\CT}_{M},p^{\CF}_{M})$ a kernel-cokernel pair.  
    %
To say that a subcategory $\CC$ of $\CA$ is \emph{closed under extensions} is to say that $\CC*\CC\subseteq\CC$. 
Conversley, if $0\in \CC$ then $\CC\subseteq \CC * \CC$  since the identity is always an admissible monic and an admissible epic. 
Note also that ${}^{\perp}(\CC^{\perp})\supseteq \CC\subseteq {}^{\perp}(\CC^{\perp})$, $ ({}^{\perp}(\CC^{\perp}))^{\perp}=\CC^{\perp} $ and $ {}^{\perp}(({}^{\perp}\CC)^{\perp})={}^{\perp}\CC$, and  for any other subcategory  $\CB\subseteq \CC$  we have $\CC^{\perp}\subseteq \CB^{\perp}$ and ${}^{\perp}\CC\subseteq {}^{\perp}\CB$. 

\subsubsection{}
\label{para-star-cats}  \cite[Corollary~8.13]{buhler-exact} Let $a\in \Hom{\CA}(M,N)$ for $M\in \CT *\CF$ and $N\in \CS*\CG$ where $\CS$, $\CG$, $\CT$ and $\CF$ 
 are subcategories of $\CA$. 
Given $p_{N}^{\CG}ai_{M}^{\CT}=0$ one obtains unique $x$ and $y$ such that the diagram below commutes
\[
\begin{tikzcd}[column sep = 1.5cm, row sep = 0.4cm]
0\arrow[r] &  {}_{\CT}M\arrow[d,  "x"']\arrow[r, "i^{\CT}_{M}"] & M\arrow[d, "a"]\arrow[r, "p_{M}^{\CF}"] & M_{\CF}\arrow[d, "y"]\arrow[r] & 0 & \exact{M}{T}{F}
\\
0\arrow[r]  &  {}_{\CS}N\arrow[r, "i^{\CS}_{N}"] & N\arrow[r, "p_{N}^{\CG}"] & N_{\CG}\arrow[r] & 0 & \exact{N}{S}{G}
\end{tikzcd}
\]
 Note that, since $\CA$ is quasi-abelian it is both idempotent complete and exact, and in this situation one can apply the usual form of the \emph{snake lemma}, where exactness refers to the concept from \ref{para-strictly-exact}. 

\subsubsection{}
\label{para-stars-are-functors} 
Write $\Hom{\CA}(\CT,\CF)=0$ to mean $\Hom{\CA}(T,F)=0$ for all $T\in \CT$ and $F\in\CF$. 
In this case any $M\in \CT * \CF$ has a $\CT$-precover $i^{\CT}_{M}\colon {}_{\CT}M\to M$ and a $\CF$-preenvelope $p_{M}^{\CF}\colon M\to M{}_{\CF}$. 
Applying the snake lemma to $a=\idfunc{M}$ in \ref{para-star-cats} defines functors of the form
    \[
    \begin{array}{cc}
        {}_{\CT}(-)\colon \CT * \CF\to \CT, & 
        (-)_{\CF}\colon \CT * \CF\to \CF.
    \end{array}
    \]
\subsection{Torsion pairs} 
\label{subsubsec-torsion}
For abelian categories, torsion pairs were introduced by Dickson \cite{Dickson-torsion-theory-for-abelian-cats}. 
We follow instead work of Tattar \cite{tattar-quasi-torion}, who worked in the more general context of quasi-abelian categories. 
Although there are more generalities to work in, they are not relevant here. 
Continue to assume $\CA$ is quasi-abelian.

\begin{defn}
    \label{defn-torsion-pair} 
    \cite[Definition~2.4]{tattar-quasi-torion} 
    A \emph{torsion pair} is a pair $(\CT,\CF)$ of subcategories with $\Hom{\CA}(\CT,\CF)=0$ and $\CT*\CF=\CA$ as in \ref{subsect-orthogonal}. 
    We call $\CT$ the \emph{torsion class}, and $\CF$ the \emph{torsion-free class}, of $(\CT,\CF)$. 
\end{defn}

\subsubsection{}
\label{para-tattar-1} 
\cite[Proposition~2.5]{tattar-quasi-torion} 
A subcategory $\CF$ is a torsion-free class in $\CA$ if and only if there is a functor $\mathsf{f} \colon \CA\to \CF$ where $\CF=\{M\mid \mathsf{f}(M)\cong M\}$ and, for all $M\in \CA$, we have $\mathsf{f}^{2}(M)\cong \mathsf{f}(M)$ and a cokernel $q_{M}\colon M\to \mathsf{f}(M)$ with $\mathsf{f}(\ker(q_{M}))=0$ and $\mathsf{f}(a)q_{M}=q_{N}a$ for all $a\in\Hom{\CA}(M,N)$. 
In this case $\mathsf{f}$ is right adjoint to the inclusion $\CF\hookrightarrow \CA$. 
Dually one can characterise torsion classes $\CT$. 

\subsubsection{}
\label{para-tattar-2} 
\cite[Proposition~2.9]{tattar-quasi-torion} 
If $\CF$ is a torsion-free class then, as in \ref{para-co-complete-covering-torsion} and \ref{subsect-orthogonal}, $\CF$ is closed under extensions, subobjects  and products.  
Dual properties hold for torsion classes. 
If $\CA$ is abelian then these properties characterise torsion and torsion-free classes; see  \cite[Theorem~2.3]{Dickson-torsion-theory-for-abelian-cats}. 
This characterisation fails in the generality of quasi-abelian categories; see \cite[Remark~2.10]{tattar-quasi-torion} for a detailed counter-example. 

\subsubsection{}
\label{para-tattar-3} \cite[Proposition~3.6]{tattar-quasi-torion} Let $(\CS,\CG)$ and $(\CT,\CF)$   both be torsion pairs in $\CA$. 
   If $a\in\Hom{\CT\cap\CG}(M,N)$ with kernel
    $h\colon K\to M$ and cokernel $d\colon N\to C$ in $\CA$, then the compositions 
   \[
   \begin{array}{ccc}
    \begin{tikzcd}
        {}_{\CT}K\arrow[r, "i_{K}^{\CT}"]
    &
    K\arrow[r, "h"] 
    &
    M
    \end{tikzcd}
    &
    {}
    &
    \begin{tikzcd}
    N\arrow[r, "d"] 
    &
    C \arrow[r, "p_{C}^{\CG}"]
    &
    C_{\CG}
    \end{tikzcd}
   \end{array}
   \]
   define the kernel and cokernel of $a$ in $\CT\cap \CG$, respectively.    In particular $\CT\cap \CG$, and dually $\CS\cap\CF$, are  preabelian.

\subsubsection{}
\label{para-tattar-4} 
\cite[Propositions~5.12~and~5.14]{tattar-quasi-torion} 
Using this observation, one can extract properties of torsion pairs in $\CA$ from similar properties for the abelian case. 
This is done via an embedding of $\CA$ into an abelian localisation of a subcategory of the homotopy category given by $2$-term complexes defined by monomorphisms; see \cite[\S5.1]{tattar-quasi-torion}. 
For example, a pair $(\CT,\CF)$ of subcategories of $\CA$ is a torsion pair if and only if we have $\CT^{\perp}=\CF$ and ${}^{\perp}\CF=\CT$. 
Furthermore, in this case, both $\CT$ and $\CF$ are again quasi-abelian by means of \ref{para-tattar-3}. 

\subsubsection{}
\label{para-tattar-5} \cite[Chapter~VI,~\S1,~Propositions~3.2~and~3.7]{stenstrom-rings-of-quotients} Consider the case where $\CA$ is abelian, and let $(\CT,\CF)$ be a torsion pair in $\CA$. 
One says $(\CT,\CF)$ is: \emph{faithful} when every projective object in $\CA$ lies in $\CF$;   \emph{split} if for each $M\in\CA$ the short exact sequence $\exact{M}{T}{F}$ splits; and \emph{hereditary} if $\CT$ is closed under subobjects. 
See for example \cite{colpi-fuller-tilting-objects} and   \cite{telpy-homo-dimension}. 
When $\CA$ is Grothendieck with enough injectives, the torsion pair $(\CT,\CF)$ is hereditary if and only if $\CF$ is closed under injective envelopes if and only if $\CT={}^{\perp}I$ for some injective object $I$.  
%


\subsection{Definable categories}
\label{subsubsec-definable
} 
Prest \cite{prest-model-theory-of-modules} introduced the notion of a \emph{definable category} which has been seen to be a natural context in which to consider pure-injectivity; see for example \cite{Prest-definable-and-monoidal}. 
Definable subcategories of a module category were characterised  by Crawley-Boevey in \cite{crawley-boevey-inf-dim-reps} using certain closure operations. 
They also have connections to the torsion theory discussed above. 
For consistency we restrict our attention to Grothendieck categories with enough injectives, and for simplicity we let $\CA$ be the category $\lMod{R}$ of left modules over a ring $R$. 
We follow a book by Jensen and Lenzing \cite{JenLen1989}.

\subsubsection{}
\label{para-jenlenz-chara-prelim} 
\cite[Theorem~6.4]{JenLen1989} 
An  embedding $\iota\colon  M\hookrightarrow N$ is said to be \emph{pure} if, for any right module $Q$,  its image $\idfunc{Q}\otimes \iota\colon Q\otimes M\to Q\otimes N$ under the functor $Q\otimes_{R}-$  is injective. 
Equivalently, for any finitely presented left module $P$, any homomorphism $P\to N/\im(\iota)$ factors through $\cok(\iota)$. 
One says $M$ is \emph{pure}-\emph{injective} if it is \emph{relatively} injective with respect to pure embeddings, meaning any pure-embedding leaving $M$ splits. 

\subsubsection{}
\label{para-jenlenz-chara} 
\cite[Theorem~7.1(vi),~Theorem~8.1(ii)]{JenLen1989} 
There are various ways to characterise pure-injective modules. 
Let $\mathtt{I}$ be an indexing set. 
        Let $\iota_{\mathtt{I}}\colon \bigoplus_{\mathtt{I}}M\to\prod_{\mathtt{I}}M$ be the inclusion of the coproduct, into the product, 
 of $M$ indexed over $\mathtt{I}$. 
         The universal property of the coproduct defines  \emph{summation} homomorphism $ \sigma_{\mathtt{I}}\colon \bigoplus_{\mathtt{I}}M\to M$ that sends $(m_{i})$ to $\sum_{i}m_{i}$. 
    Note $\sigma_{\mathtt{I}}$ makes sense since $m_{i}=0$ for all but finitely many $i$. 
It follows that  $M$ is pure-injective if and only if, for any set $\mathtt{I}$, there is a morphism $\tau_{\mathtt{I}}\colon \prod_{\mathtt{I}} M\to M$ such that $\tau_{\mathtt{I}}\iota_{\mathtt{I}}=\sigma_{\mathtt{I}}$.  
One says that $M$ is $\Sigma$-\emph{pure}-\emph{injective} if $\bigoplus_{\mathtt{I}}M$ is pure-injective for any $\mathtt{I}$, and this is equivalent to requiring that  $\iota_{\mathtt{I}}$ is a split monomorphism for any $\mathtt{I}$. 
In particular any $\Sigma$-pure injective is pure-injective. 

\begin{defn}
    \cite[\S2.6]{prest-model-theory-of-modules}  
    An additive subcategory $\CC$ of $\lMod{R}$ is (\emph{axiomatisable}, or rather) \emph{definable} if it is closed under products, directed colimits and pure submodules. 
\end{defn}

These algebraic conditions are not the original formulation of the definition, but rather a characterisation of Crawley-Boevey  \cite[\S2.3]{crawley-boevey-inf-dim-reps}, which we use to avoid introducing model theoretic background. 
Of particular relevance later will be a characterisation of definable torsion-free classes due to Angeleri H\"ugel \cite{Hugel-abundance-of-silting}. 

\subsubsection{}
\label{para-definable} 
\cite[Theorem~3.8,~Corollary~3.9]{Hugel-abundance-of-silting} 
Note that the condition of being closed under directed colimits has a relation to covering theory. 
Specifically, if $\CC$ is closed under directed colimits, then any module that has a $\CC$-precover must also have a $\CC$-cover; see \cite[Theorem~2.2.12]{xu-flat-covers-of-modules}. 
In fact, a torsion-free class $\CF$ in $\lMod{R}$ is definable if and only if every module has an $\CF$-cover, in which case there is a pure injective module $M$ such that every module in $\CF$ is a submodule of a direct product of copies of $M$.

\subsection{Morita context rings}
\label{subsec-morita-context}

Later we look at the path algebra of the Kronecker quiver over a commutative ring. 
To this end we study so-called \emph{rings of Morita context}; see for example the book by  Bass \cite{Bass-algebraic-K-theory}. 

\begin{defn}
    \cite[Chapters~II~and~III]{Bass-algebraic-K-theory} 
    A \emph{Morita context} is a tuple $(R,S;\,N,L;\,\theta,\zeta)$ where $R$ and $S$ are rings, ${}_{R}N_{S},{}_{S}L_{R}$ are bimodules,  $\theta\colon N\otimes L\to R$ and  $\zeta\colon L\otimes N\to S$ are bimodule homomorphisms and where $\theta(n\otimes \ell)n'=n\zeta(\ell\otimes n')$ and $\ell\theta(n\otimes \ell')=\zeta(\ell\otimes n)\ell'$ for each $\ell,\ell'\in L$ and $n,n'\in N$. 
It follows that there is a \emph{ring of Morita context} denoted 
$\begin{bsmallmatrix}
R & N
\\
L & S
\end{bsmallmatrix}$ 
whose ring multiplication is canonically defined by $\theta$ and $\zeta$. %
\end{defn}

\subsubsection{}
\label{para-left-mods-morita}  A left module over  $\begin{bsmallmatrix}
R & N
\\
L & S
\end{bsmallmatrix}$  is the same thing as a decorated column $\begin{bsmallmatrix}
M
\\
K
\end{bsmallmatrix}_{\psi}^{\varphi}$ where $M$ is a left $R$-module, $K$ is a left $S$-module, and  $\varphi\colon N\otimes K\to M$  and $\psi\colon L\otimes M\to K$ are module homomorphisms such that
\[
\begin{array}{cc}
\theta(n\otimes \ell)m=\varphi(n\otimes\psi(\ell\otimes m)),
&
\zeta(\ell\otimes n)k = \psi(\ell\otimes\varphi(n\otimes k)),
\end{array}
\]
for each $k\in K$, $\ell\in L$, $m\in M$ and $n\in N$. 
Here the action of $\begin{bsmallmatrix}
R & N
\\
L & S
\end{bsmallmatrix}$ is defined canonically by means of multiplying a length-$2$ column on the left. 
In case either of $\psi$ or $\varphi$ are $0$ we omit the corresponding decoration. 
For example a left module over   $\begin{bsmallmatrix}
R & 0
\\
L & S
\end{bsmallmatrix}$  has the form $\begin{bsmallmatrix}
M
\\
K
\end{bsmallmatrix}_{\psi}$ with no condition on $\psi\colon L\otimes M\to K$.

\subsubsection{}
\label{para-tensor-hom-notation} For an $S$-$R$-bimodule $L$ the image of the bijections from the tensor-hom adjunction will be denoted 
\begin{equation}
\label{tensor-hom-notation}
         \begin{array}{cc}\Hom{S}(L\otimes M,K )\leftrightarrow 
     \Hom{R}(M,\Hom{S}(L,K)),
     &
    \tens{\mu}\leftrightarrow \homs{\mu}
    \end{array}
\end{equation}

    meaning that, given either the left $S$-module homomorphism $\tens{\mu}\colon L\otimes M\to K$ or the left $R$-module homomorphism $\homs{\mu}\colon M\to \Hom{S}(L,K)$, one defines the other map using the equation $\tens{\mu}(\ell\otimes m)=(\homs{\mu}(m))(\ell)$. 

We characterise flat and injective modules for triangular rings of Morita contex  following work of Fossum, Griffith and Reiten \cite{fossum-griffith-reiten}; see also work of Haghany and Varadarajan \cite{haghany-varadarajan-formal-triangular-matrix-ring}, M{\"u}ller \cite{mueller-marianne-rings-of-quotients} and Stenstr{\"o}m \cite{Stenstrom-maximal-ring-of-quotients}.

\subsubsection{}
\label{para-flats-injectives} \cite[Corollary~1.6~(d),~Proposition~1.14~(bis.)]{fossum-griffith-reiten} Adopting the notation from \eqref{tensor-hom-notation} in \ref{para-tensor-hom-notation}, the left $\begin{bsmallmatrix}
R & 0
\\
L & S
\end{bsmallmatrix}$-module $\begin{bsmallmatrix}
M
\\
K
\end{bsmallmatrix}_{\tens{\psi}}$ is flat if and only if $M$ is a flat  $R$-module, $\cok(\tens{\psi})$ is a flat  $S$-module and $\homs{\psi}$ is injective. 
Following \cite[Theorem~3.1]{haghany-varadarajan-formal-triangular-matrix-ring} one obtains a similar result for projective left $\begin{bsmallmatrix}
R & 0
\\
L & S
\end{bsmallmatrix}$-modules.  
Likewise,  $\begin{bsmallmatrix}
M
\\
K
\end{bsmallmatrix}_{\tens{\psi}}$ is injective if and only if $K$ is an injective  $S$-module, $\ker(\homs{\psi})$ is an injective $R$-module and $\homs{\psi}$ is surjective.

\section{Closure properties for Kronecker representations}
\label{sec-linrels}


From now on  we assume that $R$ is a (unital, associative and) commutative ring. 
    Let $\Kron$ be the \emph{Kronecker quiver}: two arrows $a$ and $b$ with the same tail $t$ and same head $h$.  
    Write $R\Kron$ for the \emph{path algebra}, meaning the free left $R$-module $Re_{t}\oplus Re_{h}\oplus Ra\oplus Rb$ equipped with a multiplication $R$-linearly extending the equations 
    \[
    \begin{array}{ccccc}
    e_{t}^{2}=e_{t},
    &
    e_{h}^{2}=e_{h},
    &
    a=e_{h}a=ae_{t},
    &
    b=e_{h}b=be_{t},
    &
    e_{t}e_{h}=e_{h}e_{t}=0.
    \end{array}
    \]
    We can and do identify  $R\Kron$ with $R^{4}$ by extending $
    e_{t}\mapsto(1,0,0,0)$, $
    e_{h}\mapsto(0,1,0,0)$, $
    a\mapsto(0,0,1,0)$ and $
    b\mapsto(0,0,0,1)$ multiplicatively and $R$-linearly. 
    In this way multiplication can be described by
    \[
    \begin{array}{cc}
    (r_{t},r_{h},r_{a},r_{b})\cdot  (s_{t},s_{h},s_{a},s_{b})\coloneqq (r_{t}s_{t},r_{h}s_{h},r_{a}s_{t}+r_{h}s_{a},r_{b}s_{t}+r_{h}s_{b}),
    &
    (r_{t},s_{t},r_{h},s_{h},r_{a},s_{a},r_{b},s_{b}\in R).
    \end{array}
    \]
    Let $\lRep{R}(\Kron)$ be the category  with objects $(M_{t},M_{h};\mu_{a},\mu_{b})$ for $\mu_{a},\mu_{b}\in \Hom{R}(M_{t},M_{h})$ and  where  
           \[
           \Hom{\lRep{R}(\Kron)}((L_{t},L_{h};\lambda_{a},\lambda_{b}),(M_{t},M_{h};\mu_{a},\mu_{b}))\coloneqq \left\{(f_{t},f_{h})\left|
           \begin{array}{c}
           f_{x}\in  \Hom{R}(M_{x},N_{x})\text{ for }x=t,h,\\
        \text{such that }f_{h}\lambda_{c}=\mu_{c}f_{t}\text{ for }c=a,b.           \end{array}\right.\right\},
           \]
           and where composition and identity morphisms are defined component-wise. 
           
\begin{rem}
\label{rem-well-known-equivalences}
Considering $\Kron$ as a category, it follows that $\lRep{R}(\Kron)$ is the category of functors of the form $\Kron\to \lMod{R}$. 
    As is well-known, $\lRep{R}(\Kron)$ and $\lMod{R\Kron}$ are $R$-linear categories, and there is an $R$-linear equivalence $\lRep{R}(\Kron)\to \lMod{R\Kron}$ sending $(M_{t},M_{h};\mu_{a},\mu_{b})$ to $M_{t}\oplus M_{h}$ with the $R\Kron$-action given by
    \[
    \begin{array}{cc}
    (r_{t},r_{h},r_{a},r_{b})\cdot (m_{t},m_{h})\coloneqq  (r_{t}m_{t}, r_{a}\mu_{a}(m_{t})+r_{b}\mu_{b}(m_{t})+r_{h}m_{h}), 
    &
    (m_{t}\in M_{t},\, m_{h}\in M_{h}).
    \end{array}
    \]
    There is an $R$-algebra isomorphism between the path algebra and the lower-triangular ring of Morita context 
    \[
    \begin{array}{ccccc}
    R\Kron \to \begin{bsmallmatrix}
R & 0
\\
R^{2} & R
\end{bsmallmatrix},
&
e_{t}\mapsto \begin{psmallmatrix}
1 & 0
\\
(0,0)
&
0
\end{psmallmatrix},
e_{h}\mapsto \begin{psmallmatrix}
0 & 0
\\
(0,0)
&
1
\end{psmallmatrix},
a\mapsto \begin{psmallmatrix}
0 & 0
\\
(1,0)
&
0
\end{psmallmatrix},
b\mapsto \begin{psmallmatrix}
0 & 0
\\
(0,1)
&
0
\end{psmallmatrix}. 
    \end{array}
    \]
See for example the work of G{\"o}bel and Simson  \cite[p.~215]{Gobel-Simson-embeddings}. 
\end{rem}


\begin{lem}
\label{lem-characterising-flat-and-injective-kronecker-mods}
     For any object $(M_{t},M_{h};\mu_{a},\mu_{b})$ of $\lRep{R}(\Kron)$ the following statements hold. 
     \begin{enumerate}
         \item $(M_{t},M_{h};\mu_{a},\mu_{b})$ is flat (respectively, projective) if and only if $M_{t}$ and $\cok{\begin{psmallmatrix} \mu_{a} & \mu_{b}
\end{psmallmatrix}}$ are flat (respectively, projective) over $R$ and ${\begin{psmallmatrix} \mu_{a} & \mu_{b}
\end{psmallmatrix}}\colon M_{t}^{2}\to M_{h}$ is injective, in which case so is ${\begin{psmallmatrix} \mu_{a} \\ \mu_{b}
\end{psmallmatrix}}\colon M_{t}\to M_{h}^{2}$.   
         \item $(M_{t},M_{h};\mu_{a},\mu_{b})$ is injective if and only if $M_{h}$ and $\ker{\begin{psmallmatrix} \mu_{a} \\ \mu_{b}
\end{psmallmatrix}}$ are injective  over $R$ and ${\begin{psmallmatrix} \mu_{a} \\ \mu_{b}
\end{psmallmatrix}}\colon M_{t}\to M_{h}^{2}$ is surjective, in which case so is ${\begin{psmallmatrix} \mu_{a} & \mu_{b}
\end{psmallmatrix}}\colon M_{t}^{2}\to M_{h}$. 
     \end{enumerate}
\end{lem}

\begin{proof}
 Recall the notation from \eqref{tensor-hom-notation} in discussion \ref{para-tensor-hom-notation}. 
To begin we explain why there is a commutative diagram in $\lMod{R}$ of the form
    \[
 \begin{tikzcd}[ampersand replacement=\&, column sep = 1.5cm]
            M_{t}\arrow[rr, hook, "\Delta"]\arrow[d, "{\homs{\mu}}"']\arrow[drr, near start, "{\begin{psmallmatrix} \mu_{a} \\ \mu_{b}
\end{psmallmatrix}}"'] \&\& M_{t}^{2}\arrow[d, "{\begin{psmallmatrix} \mu_{a} & 0 \\ 0 & \mu_{b}
\end{psmallmatrix}}"']\arrow[rr, "\beta", "\simeq"']\arrow[drr, near end, "{\begin{psmallmatrix} \mu_{a} & \mu_{b}
\end{psmallmatrix}}"] \&\& R^{2}\otimes M_{t}\arrow[d, "{\tens{\mu}}"]
            \\
            \Hom{R}(R^{2},M_{h})\arrow[rr, "\simeq", "\alpha"']
            \&\& M_{h}^{2}\arrow[rr, two heads, "\Sigma"'] \&\& M_{h}
        \end{tikzcd}
    \]
    where 
for all $f\in \Hom{R}(R^{2},M_{h})$ and $m,m',m''\in M_{t}$ we let $\Delta(m'')=(m'',m'')$, $\Sigma(m,m')=m+m'$ and
  \[
  \begin{array}{cc}
    \alpha\colon \Hom{R}(R^{2},M_{h})\to M_{h}^{2},
    f\mapsto (f(1,0),f(0,1)),
       &  
    \beta \colon M_{t}^{2} \to R^{2}\otimes M_{t},
    (m,m')\mapsto (0,1)\otimes m' + (1,0)\otimes m.
  \end{array}
  \]
Now letting $\tens{\mu}\coloneqq (\mu_{a} \,\, \mu_{b})\beta^{-1}$ it is automatic that the upper-right triangle commutes. 
In fact, observing 
  \[
  \homs{\mu}(m)(r,r')=\tens{\mu}((r,r')\otimes m)=\tens{\mu}((1,0)\otimes rm+(0,1)\otimes r'm)=\tens{\mu}(\beta(rm,r'm))=\mu_{a}(rm)+\mu_{b}(r'm)
  \]
  for all $r,r'\in R$ and $m\in M_{t}$, we have $\alpha(\homs{\mu}(m))=(\mu_{a}(m),\mu_{b}(m))$ for all such $m$. 
We can observe from the commutative  diagram that $\cok{\begin{psmallmatrix} \mu_{a} & \mu_{b}
\end{psmallmatrix}}\cong \cok(\tens{\mu})$ and that $\ker{\begin{psmallmatrix} \mu_{a} \\ \mu_{b}
\end{psmallmatrix}}\cong \ker(\homs{\mu})$.  
Furthermore, in terms of \Cref{rem-well-known-equivalences}, the representation  $(M_{t},M_{h};\mu_{a},\mu_{b})$ corresponds to the left $\begin{bsmallmatrix}
R & 0
\\
R^{2} & R
\end{bsmallmatrix}$-module  $\begin{bsmallmatrix}
M_{t}
\\
M_{h}
\end{bsmallmatrix}_{\tens{\mu}}$. 

Hence the stated characterisations are the direct translations of those from \ref{para-flats-injectives}. 
Of course, assuming $\begin{psmallmatrix} \mu_{a} & \mu_{b}
\end{psmallmatrix}=\Sigma \begin{psmallmatrix} \mu_{a} & 0 \\ 0 & \mu_{b}
\end{psmallmatrix}$ is injective implies $\begin{psmallmatrix} \mu_{a} & 0 \\ 0 & \mu_{b}
\end{psmallmatrix}$ is injective, and hence that $\begin{psmallmatrix} \mu_{a} \\ \mu_{b}
\end{psmallmatrix}=\begin{psmallmatrix} \mu_{a} & 0 \\ 0 & \mu_{b}
\end{psmallmatrix} \Delta$ is injective. 
Dually if $\begin{psmallmatrix} \mu_{a} \\ \mu_{b}
\end{psmallmatrix}$ is surjective then so is $\begin{psmallmatrix} \mu_{a} & \mu_{b}
\end{psmallmatrix}$, as required. 
\end{proof}


  

The injectivity of the map  ${\begin{psmallmatrix} \mu_{a} \\ \mu_{b}
\end{psmallmatrix}}$ in \Cref{lem-characterising-flat-and-injective-kronecker-mods}(1) is key in what follows.

\begin{defn}
    \label{defn-cat-of-relations-as-a-cat-of-kron-mods}
    Let $\lRel{R}(\Kron)$ be the subcategory of $\lRep{R}(\Kron)$ consisting of objects $(M_{t},M_{h};\mu_{a},\mu_{b})$ such that the morphism  ${\begin{psmallmatrix} \mu_{a} \\ \mu_{b}
\end{psmallmatrix}}\colon M_{t}\to M_{h}^{2}$ is injective, that is, such that $\ker(\mu_{a})\cap\ker(\mu_{b})=0$. 
\end{defn}

Consider a pair of morphisms in $\lRep{R}(\Kron)$, whose composition is $0$, defining a sequence of the form 
\begin{equation}
\label{eqn-exact-seq-in-rep}
    \begin{tikzcd}        (L_{t},L_{h};\lambda_{a},\lambda_{b})\arrow[rr, "{(f_{t},f_{h})}"]
        &&
        (M_{t},M_{h};\mu_{a},\mu_{b})\arrow[rr, "{(g_{t},g_{h})}"]
        &&
        (N_{t},N_{h};\eta_{a},\eta_{b})
    \end{tikzcd}
\end{equation}
    As discussed in \Cref{rem-well-known-equivalences}, $\eqref{eqn-exact-seq-in-rep}$ is  exact if and only if the  diagrams in $\lMod{R}$ below are exact sequences
\begin{equation}
           \label{eqn-exact-seq-in-rmod}
      \begin{array}{cc}
    \begin{tikzcd}
        L_{t}\arrow[r, "{f_{t}}"]
        &
        M_{t}\arrow[r, "{g_{t}}"]
        &
        N_{t}\end{tikzcd}
        &
        \begin{tikzcd}
    L_{h}\arrow[r, "{f_{h}}"]
        &
        M_{h}\arrow[r, "{g_{h}}"]
        &
        N_{h}
            \end{tikzcd}
      \end{array}
\end{equation}

\begin{lem}
    \label{lem-exact-structure}
    Given an exact sequence of the form $\eqref{eqn-exact-seq-in-rep}$ such  that $f_{h}$ is injective, if $(L_{t},L_{h};\lambda_{a},\lambda_{b})$ and $(N_{t},N_{h};\eta_{a},\eta_{b})$ lie in the subcategory $\lRel{R}(\Kron)$ then so too does $(M_{t},M_{h};\mu_{a},\mu_{b})$. 
\end{lem}

\begin{proof}
    Fix $m\in M_{t}$ such that $\mu_{a}(m)=0$ and $\mu_{b}(m)=0$, and it remains to prove that $m=0$. 
    Since $(g_{t},g_{h})$ is a morphism in $\lRep{R}(\Kron)$ we have $\eta_{a}(g_{t}(m))=g_{h}(\mu_{a}(m))=0$ and similarly $\eta_{b}(g_{t}(m))=0$ and so, since $(N_{t},N_{h};\eta_{a},\eta_{b})$ lies in $\lRel{R}(\Kron)$, we have $g_{t}(m)=0$. 
    
    By assumption the sequences in $\eqref{eqn-exact-seq-in-rmod}$ are exact, and so $m=f_{t}(\ell)$ for some $\ell\in L_{t}$. 
     Since $(f_{t},f_{h})$ is a morphism in $\lRep{R}(\Kron)$ we have $f_{h}(\lambda_{a}(\ell))=\mu_{a}(m)=0$ and so $\lambda_{a}(\ell)=0$ since $f_{h}$ is injective.  
     Similarly $\lambda_{b}(\ell)=0$. 
     Since $(L_{t},L_{h};\lambda_{a},\lambda_{b})$ lies in $\lRel{R}(\Kron)$ this means $\ell=0$ and so $m=0$. 
\end{proof}

In \Cref{rem-small-conv-fails} we see how  the converse of \Cref{lem-exact-structure} fails. 
This follows the representation theory of the Kronecker quiver in the classical case where $R$ is a field: the exact sequence we consider is an Auslander--Reiten sequence defined by the extension of the simple injective module by the simple projective.

\begin{rem}
\label{rem-small-conv-fails}
    Let $L_{t}=N_{h}=0$, $L_{h}=M_{t}=M_{h}=N_{t}=R$ and $\mu_{a}=\mu_{b}=\idfunc{R}$. 
    Then $(L_{t},L_{h};\lambda_{a},\lambda_{b})$ and $(M_{t},M_{h};\mu_{a},\mu_{b})$ lie in $\lRel{R}(\Kron)$,  but $(N_{t},N_{h};\eta_{a},\eta_{b})$  does not lie in $\lRel{R}(\Kron)$. 
    \end{rem}

 \begin{lem}
        \label{lem-limits-and-coproducts-relations}
        The category $\lRel{R}(\Kron)$ is closed under subobjects, extensions, limits and coproducts. 
    \end{lem}

    \begin{proof}
        We firstly prove that $\lRel{R}(\Kron)$ is closed under products. 
       Let $I$ be a set, $(M_{t}(i),M_{h}(i);\mu_{a}(i),\mu_{b}(i))$ be an object in $\lRel{R}(\Kron)$ for each $i\in I$ and $(P_{t},P_{h};\rho_{a},\rho_{b})$ be their product in $\lRep{R}(\Kron)$. 
       By definition, 
       \[
       \begin{array}{cccc}
        P_{t}=\prod_{i\in I}M_{t}(i),
        &
        P_{h}=\prod_{i\in I}M_{h}(i),
        &
        \rho_{c}((m_{i}))=(\mu_{c}(i)(m_{i}))
        &
        (c=a,b)
       \end{array}
       \]
       where $m_{i}\in M_{t}(i)$ for each $i$.  
       Now suppose $\rho_{a}((m_{i}))=0$ and $\rho_{b}((m_{i}))=0$. 
       It follows that, for each $i$, we have $\mu_{a}(i)(m_{i})=0$ and $\mu_{b}(i)(m_{i})=0$, meaning that $m_{i}=0$ since $(M_{t}(i),M_{h}(i);\mu_{a}(i),\mu_{b}(i))$ lies in $\lRel{R}(\Kron)$. 
       Thus $(P_{t},P_{h};\rho_{a},\rho_{b})$ lies  in $\lRel{R}(\Kron)$, and so $\lRel{R}(\Kron)$ is closed under products. 

       We secondly check that $\lRel{R}(\Kron)$ is closed under subobjects. 
        To see this, let $(N_{t},N_{h};\eta_{a},\eta_{b})$ be an object in $\lRel{R}(\Kron)$, and let  $(g_{t},g_{h})$ be a monomorphism from an object $(M_{t},M_{h};\mu_{a},\mu_{b})$ in $\lRep{R}(\Kron)$ whose codomain is $(N_{t},N_{h};\eta_{a},\eta_{b})$. 
        Taking $(L_{t},L_{h};\lambda_{a},\lambda_{b})=0$ in \Cref{lem-exact-structure} it follows immediately that $(M_{t},M_{h};\mu_{a},\mu_{b})$ lies in $\lRel{R}(\Kron)$, and so $\lRel{R}(\Kron)$ is closed under subobjects. 

        Since being closed under limits is equivalent to being closed under products and kernels, it follows immediately that $\lRel{R}(\Kron)$ is closed under limits. 
        Likewise $\lRel{R}(\Kron)$ is closed under coproducts since they are subobjects of products. 
        By \Cref{lem-exact-structure} we also have that $\lRel{R}(\Kron)$ is extension closed. 
    \end{proof}

 By \ref{para-co-complete-covering-torsion}, \ref{para-tattar-2}  and \Cref{lem-limits-and-coproducts-relations} it follows: that every object in $\lRep{R}(\Kron)$ has a surjective $\lRel{R}(\Kron)$-envelope, and; that  $\lRel{R}(\Kron)$ is a torsion-free class in the abelian category $\lRep{R}(\Kron)$. 
       We construct this envelope explicitly in \Cref{lem-explicit-envelope}, and then we compute the corresponding torsion class in \Cref{lem-computing-torsion-class}.

    \begin{lem}
                \label{lem-explicit-envelope}
        For any $(M_{t},M_{h};\mu_{a},\mu_{b})$  in $\lRep{R}(\Kron)$, setting $K= \ker(\mu_{a})\cap \ker(\mu_{b})$ and
         \[
        \begin{array}{cccccc}
        L_{t}=M_{t}/K,
        &
        L_{h}= M_{h},
        &
        \lambda_{a}(m+K)= \mu_{a}(m),
        &
        \lambda_{b}(m+K)=\mu_{b}(m),
        &
        f_{t}(m)=m+K, 
        &
        f_{h}(m')=m'
        \end{array}
        \]
        for each $m\in M_{t}$ and $m'\in M_{h}$ defines an $\lRel{R}(\Kron)$-envelope  $(f_{t},f_{h})\colon (M_{t},M_{h};\mu_{a},\mu_{b})\to (L_{t},L_{h};\lambda_{a},\lambda_{b})$. 
    \end{lem}

\begin{proof}
    If $\lambda_{a}(m+K)=0$ and $\lambda_{b}(m+K)=0$ then $\mu_{a}(m)=0$ and $\mu_{b}(m)=0$, meaning $m\in K$ and so $m+K=0$. 
    Observe also that   $\lambda_{c}f_{t}=f_{h}\mu_{c}$ for each $c=a,b$. 
        Now let $(g_{t},g_{h})$ be a morphism of representations of the form $(M_{t},M_{h};\mu_{a},\mu_{b})\to (N_{t},N_{h};\eta_{a},\eta_{b})$ where 
        $(N_{t},N_{h};\eta_{a},\eta_{b})$ lies in $\lRel{R}(\Kron)$. 
        By definition, for any $m\in K$ this means $\eta_{a}(g_{t}(m))=g_{h}(\mu_{a}(m))=0$ and similarly $\eta_{b}(g_{t}(m))=0$, and so $g_{t}(m)=0$.         
        Thus there is a well-defined morphism $k_{t}\colon L_{t}\to N_{t}$ such that $k_{t}f_{t}=g_{t}$. 
        Setting $k_{h}=g_{h}$ then completes the construction of a morphism $(k_{t},k_{h})\colon (L_{t},L_{h};\lambda_{a},\lambda_{b})\to (N_{t},N_{h};\eta_{a},\eta_{b})$ such that $(k_{t},k_{h})(f_{t},f_{h})=(g_{t},g_{h})$. 

        This shows $(f_{t},f_{h})$ is a $\lRel{R}(\Kron)$-preenvelope. 
        Since $f_{t}$ and $f_{h}$ are both surjective, the pair $(f_{t},f_{h})$ defines an epimorphism by \Cref{lem-exact-structure}, and so $(f_{t},f_{h})$ is a $\lRel{R}(\Kron)$-envelope. 
\end{proof}

    \begin{lem}
    \label{lem-computing-torsion-class}
        An object $(M_{t},M_{h};\mu_{a},\mu_{b})$ in $\lRep{R}(\Kron)$ lies in ${}^{\perp}\lRel{R}(\Kron)$ if and only if $M_{h}=0$. 
    \end{lem}
    \begin{proof}
    Firstly suppose $M_{h}\neq0$, and let $M\coloneqq M_{h}$. 
    Let $N_{t}=M^{2}$ and $N_{h}=M$. 
    Define $\eta_{a},\eta_{b}\colon N_{t}\to N_{h}$ by $\eta_{a}(m,m')=m$ and $\eta_{b}(m,m')=m'$ for each $m,m'\in M$.  
    Hence if $\eta_{a}(m,m')=0$ and $\eta_{b}(m,m')=0$ then $(m,m')=(0,0)$, meaning $(N_{t},N_{h};\eta_{a},\eta_{b})$ lies in $\lRel{R}(\Kron)$. 
    Furthermore there is a morphism 
    $(f_{t},f_{h})\colon (M_{t},M_{h};\mu_{a},\mu_{b})\to (N_{t},N_{h};\eta_{a},\eta_{b})$ in $\lRep{R}(\Kron)$ defined by setting $f_{t}(m,m')=(\mu_{a}(m),\mu_{b}(m))$ and $f_{h}=\idfunc{M}$.  
    We have shown in this first case that $(M_{t},M_{h};\mu_{a},\mu_{b})$ cannot lie in ${}^{\perp}\lRel{R}(\Kron)$. 

    Secondly, suppose $M_{h}=0$, and let $(f_{t},f_{h})$ be an arbitrary morphism from $(M_{t},M_{h};\mu_{a},\mu_{b})$ to an object $(N_{t},N_{h};\eta_{a},\eta_{b})$ in $\lRel{R}(\Kron)$. 
    Hence $f_{h}=0$. 
    Furthermore, if $m\in M_{t}$ then $\eta_{a}(f_{t}(m))=f_{h}(\mu_{a}(m))=0$ and similarly $\eta_{b}(f_{t}(m))=0$, and since $(N_{t},N_{h};\eta_{a},\eta_{b})$ lies in $\lRel{R}(\Kron)$ this means $f_{t}(m)=0$, as required. 
    \end{proof}

Having established a torsion pair, we note when it is split in the sense discussed in \ref{para-tattar-5}.

    \begin{prop}
         \label{prop-characterising-split-torsion-pairs}
        $({}^{\perp}\lRel{R}(\Kron),\lRel{R}(\Kron))$ is split if and only if $R$ is a finite product of fields. 
    \end{prop}
    
    \begin{proof}
        Since $R$ is assumed to be a commutative ring, we note that $R$ is a finite  product of fields if and only if it is a semisimple ring by the Wedderburn--Artin theorem. 
        Given any object $(M_{t},M_{h};\mu_{a},\mu_{b})$ in $\lRep{R}(\Kron)$, by \Cref{lem-limits-and-coproducts-relations}  and \Cref{lem-explicit-envelope} there is a short exact sequence
   \[
    \xi_{(M;\mu)}\colon \begin{tikzcd}        0\arrow[r]
    &
    (M_{t}',0;0,0)\arrow[rr, "{(\iota_{t},0)}"]
        &&
        (M_{t},M_{h};\mu_{a},\mu_{b})\arrow[rr, "{(\pi_{t},\idfunc{M_{h}})}"]
        &&
        (M_{t}/M_{t}',M_{h};\overline{\mu}_{a},\overline{\mu}_{b})\arrow[r]
        &
        0
    \end{tikzcd}
    \]
    where $M_{t}'\coloneqq \ker(\mu_{a})\cap \ker(\mu_{b})$,  $\iota_{t}\colon M'_{t}\to M_{t}$ is the inclusion map, $\pi_{t}\colon M_{t}\to M_{t}/M_{t}'$ is the quotient map and, for each $x=a,b$, the map  $\overline{\mu}_{x}\colon M_{t}/M_{t}'\to M_{h}$ sends $m+M_{t}'$ to $\mu_{x}(m)$. 
    By \Cref{lem-exact-structure}, the fact that $\xi_{(M;\mu)}$ is a short exact sequence follows from the exactness of $\xi_{M}\colon 0\to M_{t}'\to M_{t}\to M_{t}/M_{t}'\to 0$. 

    On the one hand, if $R$ is semisimple then the short exact sequence $\xi_{M}$ of $R$-modules splits, giving a retract $\rho_{t}\colon M_{t}\to M_{t}'$ of $\iota_{t}$, and this defines a retract $(\rho_{t},0)$ of $(\iota_{t},0)$ in  $\lRep{R}(\Kron)$. 
    
    On the other hand, any exact sequence  $0\to L\to M\to N\to 0$  in $\lMod{R}$ is given by the kernel $L=\ker(\pi)$ of a surjective map $\pi\colon M\to N$, in which case there is an object $(M,N;\pi,\pi)$ in $\lRep{R}(\Kron)$. As above, assuming the corresponding exact sequence $ \xi_{(M,N;\pi,\pi)}$ splits then so does  $0\to L\to M\to N\to 0$. Thus if any $\xi_{(M;\mu)}$ splits then $R$ is semisimple, as required. 
    \end{proof}




Colimits in $\lRep{R}(\Kron)$ are computed pointwise, since it is a category of functors with a cocomplete target; see \Cref{rem-well-known-equivalences}. 
Hence \Cref{lem-directed-colimits-closure} follows from the fact that directed colimits: of exact sequences are exact; and commute with finite limits. 
Never-the-less, for accessibility we provide a detailed proof below.

\begin{lem}
    \label{lem-directed-colimits-closure}
    $\lRel{R}(\Kron)$ is closed under directed colimits.
\end{lem}

\begin{proof}   
Let $(I,\leq)$ be a directed set and $(M^{i}_{t},M^{i}_{h};\mu^{i}_{a},\mu^{i}_{b})$ be an object in $\lRep{R}(\Kron)$ for each $i\in I$. 
For each $i,j\in I$ with $i\leq j$ let $f^{ij}=(f^{ij}_{t},f^{ij}_{h})\in\Hom{\lRep{R}(\Kron)}(M^{i}_{t},M^{i}_{h};\mu^{i}_{a},\mu^{i}_{b}), (M^{j}_{t},M^{j}_{h};\mu^{j}_{a},\mu^{j}_{b}))$ such that $f^{ii}$ is the identity and $f^{ik}=f^{jk}f^{ij}$ whenever $i\leq j\leq k$. 
Let   $(M_{t},M_{h};\mu_{a},\mu_{b})$ be the colimit of this direct system in $\lRep{R}(\Kron)$. 
So for all  $c=a,b$ and $i,j\in I$ with $i\leq j$ we have a commutative diagram   in $\lMod{R}$ of the form
    \[
    \begin{tikzcd}[column sep = 1.5cm]
    M_{t}^{i}
    \arrow[r, 
 "{\mu_{c}^{i}}"']
    \arrow[dr, bend right, near start, "{\varphi_{t}^{i}}"']
    \arrow[rrr, bend left, near start, "{f_{t}^{ij}}"']
    &
    M_{h}^{i}\arrow[drr, bend left, near end, "{\varphi_{h}^{i}}"']
    \arrow[rrr, bend left, near end, "{f_{h}^{ij}}"']
    &
    &
    M_{t}^{j}
    \arrow[r, 
"{\mu_{c}^{j}}"']
    \arrow[dll, bend right, near end, "{\varphi_{t}^{j}}"]
    &
    M_{h}^{j}\arrow[dl, bend left, near start, "{\varphi_{h}^{j}}"]
    \\
    &
    M_{t}
\arrow[rr, 
"{\mu_{c}}"]
    &
    &
    M_{h}
    &
    \end{tikzcd}
    \]
   Since colimits in $\lRep{R}(\Kron)$ are computed pointwise in $\lMod{R}$, for each $x=h,t$ we have $M_{x}= (\bigoplus_{i} M_{x}^{i})/N_{t}$ 
where $(m^{i})$ lies in $N_{x}$ 
precisely when there exists $k\in I$ with $m^{i}=0$ for $i<k$ and $m^{i}=f_{x}^{ki}(m^{k})$ 
for $i\geq k$.  
    Furthermore if $\ell\in M_{x}^{j}$ then $\varphi_{x}^{j}(\ell)=(m^{i})+N_{x}$ where $m^{j}=\ell$ and $m^{i}=0$ for $i\neq j$. 
Thus by the commutativity of the diagram we have $\mu_{c}((m^{i})+N_{t})=(\mu_{c}^{i}(m^{i}))+N_{h}$ for each $c=a,b$. 

    We now assume $(M_{t}^{i},M_{h}^{i};\mu_{a}^{i},\mu_{b}^{i})$ defines an object in $\lRel{R}(\Kron)$, meaning $\ker(\mu_{a}^{i})\cap\ker(\mu_{b}^{i})=0$, for each $i$. 
    From here, to complete the proof it suffices to prove that $(m^{i})\in N_{t}$ assuming $(m^{i})\in \ker(\mu_{a})\cap\ker(\mu_{b})$. 
    The assumption says that, for each $c=a,b$, there exists $k(c)\in I$ such that $\mu_{c}^{i}(m^{i})=0$ for $i<k(c)$ and 
\[
\mu_{c}^{i}(m^{i})=f_{h}^{k(c)i}(\mu_{c}^{k(c)}(m^{k(c)}))=\mu_{c}^{i}(f_{t}^{k(c)i}(m^{k(c)}))\text{, and so }m^{i}-f_{t}^{k(c)i}(m^{k(c)})\in \ker(\mu_{c}^{i}),
\]
 for $i\geq k(c)$. 
    Without loss of generality, $k(a)\geq k(b)$. 
    From here we claim 
    \[
    \begin{array}{ccc}
    \text{(1) }m^{i}=0\text{ if }i<k(b),
    &
    \text{(2) }m^{i}=f_{t}^{k(b)i}(m^{k(b)})\text{ if }k(b)\leq i <k(a),
    &
    \text{(3) }m^{i}=f_{t}^{k(a)i}(m^{k(a)})\text{ if }k(a)\leq i. 
    \end{array}
    \]
If we show claims (1), (2) and (3) hold then it follows that $(m^{i})=(m^{i}_{b})+(m^{i}_{a})$ where $(m^{i}_{b}),(m^{i}_{a})\in\bigoplus_{i}M_{t}^{i}$ are defined by $m^{i}_{b}\coloneqq0$ for $i<k(b)$, $m^{i}_{b}\coloneqq f_{t}^{k(b)i}(m^{k(b)})$ for $i\geq k(b)$,  $m^{i}_{a}\coloneqq0$ for $i<k(a)$ and $m^{i}_{a}\coloneqq f_{t}^{k(b)i}(m^{k(a)}-f_{t}^{k(b)k(a)}(m^{k(b)}))$ for $i\geq k(a)$. 
Thus to complete the proof we just check (1), (2) and (3). 

For each equation it is necessary and sufficient to prove that the difference of the given expressions lies in $\ker(\mu_{a}^{i})\cap\ker(\mu_{b}^{i})=0$. 
Recall, for each $c=a,b$, that $m^{i}-f_{t}^{k(c)i}(m^{k(c)})\in \ker(\mu_{c}^{i})$ for $i\geq k(c)$.

(1) If $i<k(b)$ then also $i<k(a)$ and so  $\mu_{b}^{i}(m^{i})=0=\mu_{b}^{i}(m^{i})$. 

(2) If $i<k(a)$ then $m^{i}\in \ker(\mu_{a}^{i})$ and also $\mu^{i}_{a}f_{t}^{k(b)i}=f^{k(b)i}_{h}\mu^{k(b)}_{a}$ and so $f_{t}^{k(b)i}(m^{k(b)})\in \ker(\mu_{a}^{i})$. 

(3) Note that $m^{i}-f_{t}^{k(a)i}(m^{k(a)})=x+f_{t}^{k(a)i}(y)$ where $x=m^{i}-f_{t}^{k(b)i}(m^{k(b)})$ and $y=f_{t}^{k(b)k(a)}m^{k(b)}-m^{k(a)}$, and when $k(b)\leq k(a)\leq i$ we have $x\in \ker(\mu^{i}_{b})$, $\mu^{i}_{b}f_{t}^{k(a)i}=f_{h}^{k(a)i}\mu_{b}^{k(a)}$ and $y\in \ker(\mu_{b}^{k(a)}$ since $k(b)\leq k(a)$.  
\end{proof}

        By \Cref{rem-small-conv-fails} we have that  $\lRel{R}(\Kron)$ is never closed under quotients, and by \ref{para-tattar-2} this means that defining monomorphic $\lRel{R}(\Kron)$-covers of arbitrary objects in $\lRep{R}(\Kron)$ is impossible. 
        Never-the-less, as we shall see in \Cref{thm-main-general}, $\lRel{R}(\Kron)$ is still covering.

\begin{thm}
    \label{thm-main-general}
    As an $R$-linear subcategory of $\lMod{R\Kron}$, $\lRel{R}(\Kron)$ is: precovering and enveloping; closed under extensions,  limits, filtered colimits and coproducts; and a definable subcategory of $\lMod{R\Kron}$ containing every flat $R\Kron$-module. 
Consequently $\lRel{R}$ is the torsion-free class in a faithful hereditary torsion theory, and in particular, it is covering and quasi-abelian (and hence exact). 
\end{thm}

\begin{proof}
By \Cref{lem-limits-and-coproducts-relations},  $\lRel{R}(\Kron)$ is closed under subobjects, extensions, limits and coproducts. 
    By \Cref{lem-explicit-envelope}, $\lRel{R}(\Kron)$  is enveloping. 
    By the discussion in \ref{para-tattar-2}  this means $\lRel{R}(\Kron)$ is a torsion-free class. 
    By \Cref{lem-directed-colimits-closure} $\lRel{R}(\Kron)$ is closed under directed colimits. 
By the discussion in \ref{para-definable}, this means $\lRel{R}(\Kron)$ is  definable, and hence also that $\lRel{R}(\Kron)$ is covering.   
    By \Cref{lem-characterising-flat-and-injective-kronecker-mods} this torsion-free class contains every flat $R\Kron$-module, and so in particular it is faithful in the sense of \ref{para-tattar-5}. 
    By \Cref{lem-computing-torsion-class} the corresponding torsion class consists of the representations supported only at $t$, and since this is closed under submodules, the corresponding torsion theory is hereditary.  
    By the discussion in \ref{para-tattar-4} this means $\lRel{R}(\Kron)$ is quasi-abelian and, thus, as in \ref{para-strictly-exact}, it has an exact structure given by the class of all kernel-cokernel pairs. 
\end{proof}

\section{Reductions of relations over local rings}
\label{subsec-local-situation}

As before $R$ is a commutative ring. 
Our main motivation for studying $\lRel{R}(\Kron)$, as we did in \Cref{thm-main-general}, was that it is equivalent to  a \emph{category of relations}, defined as follows; see \Cref{lem-essential-image}. 
%


\begin{defn}
        \label{defn-category-of-relations}
        For left $R$-modules $L$ and $M$, an \emph{$R$-linear relation from}  $L$ \emph{to} $M$ is an $R$-submodule $C$ of the direct sum $L\oplus M$. 
An $R$-linear relation from $M$ to $M$ is referred to as an $R$-linear relation \emph{on} $M$.

        The category $\lRel{R}$ of $R$-linear relations is defined as follows. 
         The objects of $\lRel{R}$ are  pairs $(M,C)$ where $C$ is a relation on a left $R$-module  $M$. 
         Given a pair $(L,B),(M,C)$ of such objects we let 
         \[
         \Hom{\lRel{R}}((L,B),(M,C))\coloneqq \{\langle f \rangle \mid f\in \Hom{R}(L,M) \colon (f(x),f(y))\in C\text{ for all }(x,y)\in B\}.
         \]
         We emphasise here that the brackets around $\langle f\rangle $ exist only to distinguish it from $f$.  
         The composition of morphisms $\langle f\rangle\colon (L,B)\to (M,C)$ and $\langle g \rangle\colon (M,C)\to (N,D)$ is defined by $\langle g \rangle\circ \langle f\rangle \coloneqq \langle gf\rangle$, which makes sense since for each $(x,y)\in B$ we have $(f(x),f(y))\in C$ and hence $(g(f(x)),g(f(y))\in D$. 
         The identity morphisms in $\lRel{R}$ with respect to this composition are defined by $\idfunc{(M,C)}\coloneqq \langle \idfunc{M}\rangle$. 
       \end{defn}

    

\begin{lem}
\label{lem-existence-of-functor}
\label{lem-essential-image}
    There exists an  $R$-linear equivalence between $\lRel{R}$ and $\lRel{R}(\Kron)$. 
\end{lem}

\begin{proof}
We begin by establishing a fully-faithful functor  $\nab{R}\colon \lRel{R}\rightarrow \lRep{R}(\Kron)$. 
Define $\nab{R}$ on objects $(M,C)$ by $(C,M;\alpha_{(M,C)},\beta_{(M,C)})$ where $\alpha_{(M,C)}(x,y)=x$ and $\beta_{(M,C)}(x,y)=y$ for each $(x,y)\in C$. 
Let $\langle f \rangle\in \Hom{\lRel{R}}((L,B),(M,C))$ for some  $R$-linear map $f\colon L\to M$ with $(f(x),f(y))\in C$ for all $(x,y)\in B$. 

We now define the morphism
\[
\nab{R}( \langle f\rangle )\coloneqq (f_{\Delta},f) \in \Hom{\lRep{R}(\Kron)}((B,L;\alpha_{(L,B)},\beta_{(L,B)}),(C,M;\alpha_{(M,C)},\beta_{(M,C)})).
\] 
Declare $f_{\Delta}\colon B\to C$ to be the $R$-linear map  given by $(x,y)\mapsto (f(x),f(y))$. 
To see that $\nab{R}( \langle f\rangle )$ is a morphism in $\lRep{R}(\Kron)$, consider that for each $(x,y)\in C$ we have 
\[
\begin{array}{cc}
     f(\alpha_{(L,B)}(x,y))=f(x)=\alpha_{(M,C)}(f_{\Delta}(x,y)),
     & 
     f(\beta_{(L,B)}(x,y))=f(y)=\beta_{(M,C)}(f_{\Delta}(x,y)).
\end{array}
\]
It is clear $\nab{R}$ respects composition and identities.  
For any $r\in R$ we have $rf_{\Delta}=(rf)_{\Delta}$ since $r(f(x),f(y))=((rf)(x),(rf)(y))$ for all $(x,y)\in B$. 
Thus we have $r\nab{R}(\langle f\rangle)=\nab{R}(\langle rf\rangle)$ giving that $\nab{R}$ is $R$-linear. 

To see that $\nab{R}$ is full, for any morphism $(f_{t},f_{h})\colon (B,L;\alpha_{(L,B)},\beta_{(L,B)})\to (C,M;\alpha_{(M,C)},\beta_{(M,C)})$,
\[
f_{t}(x,y)=(\alpha_{(M,C)}(f_{t}(x,y)),\beta_{(M,C)}(f_{t}(x,y)))
=(f_{h}(\alpha_{(L,B)}(x,y)),f_{h}(\beta_{(L,B)}(x,y)))=(f_{h}(x),f_{h}(y)),
\]
meaning that $(f_{h})_{\Delta}=f_{t}$, and so $(f_{t},f_{h})=\nab{R}(\langle f_{h}\rangle)$. 
Since $\nab{R}$ is $R$-linear, it is straightforward to see that $\nab{R}$ is also faithful, since the zero morphism on $(M_{t},M_{h};\mu_{a},\mu_{b})$ is given by $(0_{M_{t}},0_{M_{h}})$. 

From here it suffices to show that 
     $\lRel{R}(\Kron)$ is the essential image   of  $\nab{R} $. 
    Suppose we are given an arbitrary isomorphism $(\theta_{t},\theta_{h})\colon (L_{t},L_{h};\lambda_{a},\lambda_{b})\to(C,M;\alpha_{(M,C)},\beta_{(M,C)})$. 
    By \Cref{lem-existence-of-functor}, and since $\theta$ is a morphism in the category $\lRep{R}(\Kron)$, for each $\ell\in L_{t}$ we have
    \[
    \begin{array}{ccc}
    \theta_{t}(\ell)=(\alpha_{(M,C)}(\theta_{t}(\ell)),\beta_{(M,C)}(\theta_{t}(\ell))),
    &
    \alpha_{(M,C)}(\theta_{t}(\ell))=\theta_{h}(\lambda_{a}(\ell)),
    &
    \beta_{(M,C)}(\theta_{t}(\ell))=\theta_{h}(\lambda_{b}(\ell)). 
    \end{array}
    \]
    Now suppose  $\lambda_{a}(\ell)=0$ and $\lambda_{b}(\ell)=0$. 
    Using the equations above this means that $\theta_{t}(\ell)=0$, meaning that $\ell=0$ since $\theta_{t}$ is injective. 
    On the other hand, sending $(x,y)$ to $(x,y)=(\alpha_{(M,C)}(x,y),\beta_{(M,C)}(x,y))$ clearly defines an injective map of the form $C\to M\oplus M$.  
\end{proof}

For the convenience of the reader we recall the notation from \eqref{eqn-splitting-thm-submodules} in the introduction. 

\begin{notn}
\label{notn-generalised-image}
\label{notn-inverting-linear-relations}
\label{notn-composing-linear-relations}
\label{notation-for-relations}
    Let $C$ be an $R$-linear relation from $L$ to $M$. 
    For any $\ell\in L$  and any subset $S\subseteq L$ let 
    \[
        \begin{array}{cc}
    C\ell\coloneqq\{m\in M\colon (\ell,m)\in C\},
    &
    CS\coloneqq \bigcup_{\ell\in S}C\ell=\{m\in M\colon (\ell,m)\in C\text{ for some }\ell\in S\}. 
        \end{array}
    \]
Let $D$ be an $R$-linear relation from $M$ to 
to $N$. 
The \emph{composition} $DC$ is the relation from $L$ to
$N$ given by
\[
DC\coloneqq \{(\ell,n)\in L\oplus N\colon\text{there exists } m\in M \text{ with } (m,n)\in D,\,(\ell,m)\in C\}
\]
The \emph{inverse} of $C$  is the $R$-linear relation $C^{-1}\coloneqq \{(m,\ell)\in M\oplus L\colon(\ell,m)\in C\}$. 
When $L=M$ let 
   \[
   \begin{array}{c}
C'' = \{ m\in M \colon  \text{$\exists \,(m_{n})\in M^{\mathbb{N}}$ with $m_{n+1}\in Cm_{n}$ for all $n$ and $m=m_{0}$}\},
\\
C' = \{ m\in M \colon  \text{$\exists \,(m_{n})\in M^{\mathbb{N}}$ with $m_{n+1}\in Cm_{n}$ for all $n$, $m=m_{0}$ and  $m_{n} = 0$ for $ n\gg0$}\},
    \\
       \begin{array}{cc}
 C^{\sharp} = C'' \cap (C^{-1})'',
 &
 C^{\flat} = C'' \cap (C^{-1})'+(C^{-1})'' \cap C'.
        \end{array}
        \end{array}
   \]    
\end{notn}

\begin{example}
\label{example-thesis-relation} 
\label{example-graph-lienar-relation}
(i) If $f\colon L\to M$ is an $R$-linear map then its \emph{graph} is the $R$-linear relation $\graph{f}\coloneqq\{(\ell,f(\ell))\mid \ell\in L\}$ from $L$ to $M$. 
In this case  $CS$ is the image of $S$ under $f$. 
    Conversley, if $C$ is an $R$-linear relation from $L$ to $M$ such that $C0=0$, then the assignment $\ell\mapsto m$ if and only if $m\in C\ell$ defines an $R$-linear map $f\colon L\to M$ with $C=\graph{f}$.   
    
    If $f\colon L\to M$ and $g\colon M\to N$ are $R$-linear maps, and if $C$ is the graph of $f$ and if $B$ is the graph of $g$, then $BC$ is the graph of $gf$. 
    Such examples, and their compositions, are prototypical. 

(ii) For a different example, let $I$ be an ideal in $R$ and let $M$ be the quotient $F/K$ of the free module $F=Rz_{0}\oplus Rz_{1}$ by the submodule $K=Iz_{1}$, and for each $i=0,1$ let $\overline{z}_{i}=z_{i} + K$. 
Define the $R$-linear map $f\colon M\to M$ by $f(r\overline{z}_{0} + s\overline{z}_{1})=r\overline{z}_{1}$, and let $C$ be the graph of $f$. 
Since $f^{2}=0$ it follows that $CC=M\oplus 0$. 
Let $D=C^{-1}C$ considered as an $R$-linear relation on $M$. 
By definition, if $r,r',s,s'\in R$ and
\[
\begin{array}{cc}
m=r\overline{z}_{0} + s\overline{z}_{1},
&
m'=r'\overline{z}_{0} + s'\overline{z}_{1}
\end{array}
\]
then $(m,m')\in D$ if and only if  $f(m)=f(m')$ if and only if $(r-r')\overline{z}_{1}=0$ if and only if $r-r'\in I$. 
On one hand, if $I=R$ then $D=M\oplus M\cong R\oplus R$. 
On the other hand, if $I=0$ then $K=0$ and $D\cong R\oplus R\oplus R$. 

(iii)
Let $R$ be a field and $C$ be an $R$-linear relation. 
The subspace $C'$ is equal to the \emph{stable kernel} $\bigcup _{n>0}C^{n}0$, and $C''$ is a subspace of the \emph{stable image} $\bigcap _{n>0}C^{n}V$, as defined by Ringel \cite[\S2]{Rin1975}. 
Furthermore if $\mathrm{dim}_{k}(M)<\infty$ then  $C''$ is equal to the stable image; see \cite[Lemma~4.2]{Cra2018}. 
\end{example}

We unpack \Cref{notation-for-relations} a little. 

\begin{rem}
\label{rem-unpacking-relation-notn}
    Let $m\in C^{\sharp}$ and so in particular $m\in C''$ giving a sequence $(m_{n})\in M^{\mathbb{N}}$ with $m_{n+1}\in Cm_{n}$ for all $n\in\mathbb{N}$ and $m=m_{0}$. 
    Hence for any $r\in R$ the diagonal $R$-linear action on $M\oplus M$ gives $(rm_{n},rm_{n++1})=r(m_{n},m_{n+1})\in C$, and so $C'\subseteq C''$ are both $R$-submodule of $M$ by restriction.  
    
Since $m\in (C^{-1})''$ we also similarly have a sequence $(m'_{n})\in M^{\mathbb{N}}$ with $m'_{n+1}\in C^{-1}m'_{n}$ for all $n\in\mathbb{N}$ and $m=m'_{0}$. 
Now let $m_{-n}\coloneqq m_{n}'$ meaning $m_{n+1}\in Cm_{n}$ for all $n\in\mathbb{Z}$. 
In other words $C^{\sharp}$ is the set of elements $m_{0}$ arising in the middle term of a sequence $(m_{n})\in M^{\mathbb{Z}}$ with $m_{n+1}\in Cm_{n}$ for all $n\in\mathbb{Z}$. 
Similarly $C^{\flat}$ consists of the sums $m^{+}_{0}+m^{-}_{0}$ of the middle terms of such sequences $(m^{\pm}_{n})$ where $m_{n}^{\pm }=0$ for $\pm n\gg 0$. 
\end{rem}

The following result from \cite{Cra2018} was written only in the context where $R$ is a field. The proof does not make use of this assumption, and generalises with no complication. 
For completeness we expose this fact. 

\begin{lem}
\label{lem-auto-of-sharp-over-flat}
\emph{\cite[Lemmas~4.4~and~4.5]{Cra2018}} For any object $(M,C)$ of $\lRel{R}$ we have
\[
\begin{array}{cccc}
C^{\sharp}\subseteq CC^{\sharp},
&
C^{\flat}=C^{\sharp}\cap CC^{\flat},
&
C^{\sharp}\subseteq C^{-1}C^{\sharp},
&
C^{\flat}=C^{\sharp}\cap C^{-1}C^{\flat}.
\end{array}
\]
Consequently there is an $R$-module automorphism $\theta$ on $C^{\sharp}/C{}^{\flat}$
defined by 
\[
\text{$\theta(m+C^{\flat})=m'+C^{\flat}$ if and only if $m'\in C^{\sharp}\cap(C^{\flat}+Cm)$.}
\]
\end{lem}

\begin{proof}

In what remains in the proof we use the notation from \Cref{rem-unpacking-relation-notn} without further reference. 

Firstly, setting $\ell_{n}\coloneqq m_{n+1}$ for all $n\in\mathbb{Z}$ gives $\ell_{n+1}=m_{n+2}\in Cm_{n+1}=C\ell_{n}$ and so $m_{1}\in C^\sharp$ meaning $m=m_{0}\in Cm_{1}\subseteq CC^{\sharp}$ giving $C^{\sharp}\subseteq CC^{\sharp}$. 
Using that $C^{\sharp}=(C^{-1})^{\sharp}$ gives $C^{\sharp}\subseteq C^{-1}C^{\sharp}$. 

Secondly, suppose we also have $m_{0}=m_{0}^{+}+m_{0}^{-}$ where $m_{n}^{\pm }=0$ for $\pm n\gg 0$. 
In this situation we likewise write $\ell^{\pm}_{n}=m_{n+1}^{\pm}$ for all $n$, meaning in particular that $\ell_{n}^{\pm }=0$ for $\pm (n+1)\gg 0$, and so for $\pm n\gg 0$. 
It follows that $\ell_{0}^{\pm }\in C((C^{\mp 1})'\cap (C^{\pm 1})'')$ and so $\ell_{0}\in CC^{\flat}$ since $C^{\flat}$ is an $R$-submodule, and so an additive subgroup. 

This argument shows that $C^{\flat}\subseteq C^{\sharp}\cap CC^{\flat}$ and for the converse we assume $k\in C^{\sharp}\cap CC^{\flat}$ and suppose $k\in Cm$ for $m\in C^{\flat}$ as above. 
We now have $m^{\pm}_{-1}\in Cm^{\pm}_{0}$, and taking the sum gives $m'\in Cm$ where $m'\coloneqq m^{+}_{-1}+m^{-}_{-1}$, which taking the difference with $k$ gives $k-m'\in C0$ as above we have $m'\in C^{\flat}$ and so in particular $k-m'\in (C^{-1})''\cap C'\subseteq C^{\flat}$ as required for the equality $C^{\flat}=C^{\sharp}\cap CC^{\flat}$.  
As above, using that  $C^{\sharp}=(C^{-1})^{\sharp}$ and that  $C^{\flat}=(C^{-1})^{\flat}$ gives $C^{\flat}=C^{\sharp}\cap C^{-1}C^{\flat}$. 

We now prove that the given formula defines an $R$-module automorphism. 
We claim that the formula ensures $\theta$ is well-defined. 
So let $m+C^{\flat}=\ell+C^{\flat}$ for some $m,\ell\in C^{\sharp}$ and suppose $m'\in C^{\sharp}\cap(C^{\flat}+Cm)$ and $\ell'\in C^{\sharp}\cap(C^{\flat}+C\ell)$. 
Hence there exist  $k,j\in C^{\flat}$ such that $(m,m'-k),(\ell,\ell'-j)\in C$. 
This gives $C^{\sharp}\ni\ell'-m'+j-k\in C(\ell-m)\in CC^{\flat}$ and so $\ell'-m'\in C^{\flat}$ by the equality $C^{\flat}=C^{\sharp}\cap CC^{\flat}$ giving the claim. 

To see that $\theta$ is surjective, if $m'\in C^{\sharp}\subseteq CC^{\sharp}$ we have $m'\in Cm$ for $m \in C^{\sharp}$ and so $\theta(m+C^{\flat})=m'+C^{\flat}$. 

To see that $\theta$ is injective, note that if $m'\in C^{\flat}\cap (C^{\flat}+Cm)$ then $m\in C^{-1}C^{\flat}$ giving $m\in C^{\sharp}\cap C^{-1}C^{\flat}= C^{\flat}$. 

 Finally let $r\in R$. 
 Then if $m'\in C^{\sharp}\cap(C^{\flat}+Cm)$, say $m'=k+\ell$ for $k\in C^{\flat}$ and $\ell\in Cm$, then $rk\in C^{\flat}$ and also $(rm,r\ell)\in C$ giving $rm'\in  C^{\sharp}\cap(C^{\flat}+C(rm))$, meaning that $\theta$ is $R$-linear. 
\end{proof}


\begin{lem}
     \label{lem-sharp-over-flat-functorial} The following assignments define $R$-linear functors
    \[
    \begin{array}{cc}
\lRel{R}\overset{\sharp\,/\,\flat}{\xrightarrow{\hspace{0.5cm}}} \lMod{R[T,T^{-1}]},
    \,
(M,C)\xmapsto{\hspace{0.5cm}} C^{\sharp}/C^{\flat},
&
 \lMod{R[T,T^{-1}]}\overset{\mathrm{graph}_{T}}{\xrightarrow{\hspace{0.5cm}}} \lRel{R},
    \,
    X\xmapsto{\hspace{0.5cm}}(X,\graph{T_{X}}),
    \end{array}
    \]
    where $T_{X}\in \End{R}(X)$ is defined by $x\mapsto Tx$. 
Furthermore, $(\sharp\,/\,\flat)\circ\mathrm{graph}_{T}$ is the identity on $\lMod{R[T,T^{-1}]}$.
\end{lem}

\begin{proof}
    By \Cref{lem-auto-of-sharp-over-flat} the $R$-module $C^{\sharp}/C^{\flat}$ has the structure of an $R[T,T^{-1}]$-module. 
    For a morphism $\langle f\rangle\colon (L,B)\to (M,C)$ in we claim that $f(\ell)$ lies in $C^{\sharp}$ (respectively, $C^{\flat}$) for any $\ell$ lying in $B^{\sharp}$ (respectively, $B^{\flat}$). 
    Indeed, following the notation from \Cref{rem-unpacking-relation-notn}, provided a sequence $\ell_{n}\in L^{\mathbb{Z}}$ with $\ell_{n}\in B \ell_{n+1}$ for all $n$ one has that  $f(\ell_{n})\in C f(\ell_{n+1})$ and so $\ell\in B^{\sharp}$ implies $f(\ell)\in C^{\sharp}$ and, since $f$ is $R$-linear, that $\ell\in B^{\flat}$ implies $f(\ell)\in C^{\flat}$ and thus the claim holds. 
    
    Let $\tau$ and $\theta$ be the automorphisms of $B^{\sharp}/B^{\flat}$ and $C^{\sharp}/C^{\flat}$ induced by $B$ and $C$ respectively.  
    Now suppose $\tau(\ell+B^{\flat})=\ell'+B^{\flat}$ which by \Cref{lem-auto-of-sharp-over-flat} means $\ell'\in B^{\sharp}\cap(B^{\flat}+B\ell)$, and so there is some $\ell''\in B^{\flat}$ such that $\ell''-\ell'=\ell$. 
    By what we claimed above, we have $f(\ell'')\in C^{\flat}$ and $f(\ell')\in C^{\sharp}$ and so $f(\ell')\in C^{\sharp}\cap(C^{\flat}+Cf(\ell))$. 
    
    Hence by \Cref{lem-auto-of-sharp-over-flat} we have that $\theta(f(\ell)+C^{\flat})=f(\tau(\ell))+C^{\flat}$ and from here functoriality is straightforward. 
    Now let $C$ be the graph of the $R$-linear automorphism $T$ on $X$ then for any element $x_{0}\in X$ one can define a sequence $(x_{n})\in X^{\mathbb{Z}}$ by $x_{n}=T^{n}x_{0}$, by construction giving $(x_{n},x_{n+1})\in C$. 
    Hence by \Cref{rem-unpacking-relation-notn} it follows that $C^{\sharp}=X$. 
    Furthermore, if $x_{n}=0$ for some $n$ then $x_{0}=T^{-n}x_{n}=0$ and so $C^{\flat}=0$. 
\end{proof}

\begin{defn}
\label{def-reductions}
\cite[Definition 1.4.32]{Ben2018} 
A \emph{reduction} of an object $(M,C)$ of $\lRel{R}$  is a 
pair $(X\mid\rho)$ where $X$ is an $R[T,T^{-1}]$-module and  $\rho\colon X\to M$ is an $R$-linear map such that 
\[
\begin{array}{cc}
C^{\sharp}=C^{\flat}+\im(\rho),
     & 
     \rho(Tx)\in C\rho(x) \quad(x\in X).
\end{array}
\]
We say a reduction $(X\mid \rho)$ of $C$ is \emph{free} if $X$ is free as an $R$-module. 
We say $(X\mid \rho)$ \emph{meets in the radical} if the preimage $\rho^{-1}(C^{\flat})=\{x\in X\colon \rho(x)\in C^{\flat}\}$ of $C^{\flat}$ coincides with radical $\rad(X)$ of $X$ considered as an $R$-module, meaning the intersection of all maximal $R$-submodules. 
\end{defn}

\Cref{thm-existence-of-reductions} provides sufficient conditions for the existence of a reduction. 
Before stating and proving this result we note a consequence of \Cref{lem-sharp-over-flat-functorial} for \Cref{def-reductions}  in \Cref{lem-reductions-define-surjective-automorphism-respecting-maps}. 

\begin{lem}
    \label{lem-reductions-define-surjective-automorphism-respecting-maps} 
    If $(X\mid \rho)$ is a reduction of $(M,C)$ then  $\langle \rho\rangle \colon \mathrm{graph}_{T}(X)\to (M,C)$ is a morphism in $\lRel{R}$. 
Furthermore if $(X\mid \rho)$ meets in the radical then there is a short exact sequence in $\lMod{R[T,T^{-1}]}$  given by
    \[
    \begin{tikzcd}
        0\arrow[r]
        &\rad(X)\arrow[r, "\subseteq"]
        & X\arrow[rr, "{(\sharp/\flat)\langle\rho\rangle}"]
        &
        & C^{\sharp}/C^{\flat}\arrow[r]
        & 0.
    \end{tikzcd}
\]
\end{lem}

\begin{proof}
    Since $\rho(Tx)\in C\rho(x)$ for all $x\in X$ it follows that $(x,y)\in \graph{T_{X}}$ implies $(\rho(x),\rho(y))\in C$ and so  $\langle \rho\rangle$ is a morphism.  
    Applying the functor $\sharp/\flat$, by \Cref{lem-sharp-over-flat-functorial} there is an $R[T,T^{-1}]$-module homomorphism $ (\sharp/\flat)\langle\rho\rangle\colon X\to C^{\sharp}/C^{\flat}$ given by the assignment $x\mapsto \rho(x)+C^{\flat}$ and so the kernel is precisely $\{x\in X\colon \rho(x)\in C^{\flat}\}$. 
    Since $C^{\sharp}=C^{\flat}+\im(\rho)$ it follows that this $R[T,T^{-1}]$-module homomorphism is surjective, and to say that $(X\mid \rho)$ meets in the radical is to say that the kernel  coincides with $\rad(X)$.
\end{proof}


\begin{prop}
    \label{prop-reductions-generalise-relations-that-split}
    If $R$ is semisimple then the following statements are equivalent for any $(M,C)$   in $\lRel{R}$. 
    \begin{enumerate}
        \item There exists a reduction $(X\mid \rho)$ of $(M,C)$ that meets in the radical. 
        \item We have $C^{\sharp}=C^{\flat}\oplus Y$ for an $R[T,T^{-1}]$-module $Y$ with $Ty=z$ if and only if $z\in Cy$ for $y,z\in Y$. 
    \end{enumerate}
\end{prop}

\begin{proof}
$(1\Rightarrow 2)$ 
Since $R$ is a semisimple ring the inclusion of $R$-modules $C^{\flat}\subseteq C^{\sharp}$ has a complement $Y$. 
Let $\tau=(\sharp/\flat)\langle\rho\rangle$. 
That $R$ is semisimple also implies that $\rad(X)=0$, and so  $\tau \colon X\to C^{\sharp}/C^{\flat}$ is an isomorphism of $R[T,T^{-1}]$-modules by \Cref{lem-reductions-define-surjective-automorphism-respecting-maps}. 
In particular, $\tau$ is injective, and recalling that $\tau(x)=\rho(x)+C^{\flat}$, it follows that $\im(\rho)\cap C^{\flat}=0$  and that $\rho$ is injective, for otherwise there would be $x\neq 0$ with $\tau(x)=0$. 
Now let $Y=\im(\rho)$, and write $\sigma$ for the $R$-module isomorphism $X\to Y$ sending $x\mapsto \rho(x)$. 

Since $X$ is an $R[T,T^{-1}]$ so is $Y$ by means of the formula $Ty\coloneqq \sigma (T( \sigma^{-1}(y)))$. 
On the one hand, if $y\in Y$, say of the form $y=\rho(x)$ for some $x\in X$, this means $x=\sigma^{-1}(y)$ and  so if $z=Ty$ then $z=\rho(Tx)\in C\rho(x)=Cy$. 
On the other hand, if $z\in Cy$ for $y,z\in Y$ then writing $y=\rho(x)=\sigma(x)$ we have $C^{\sharp}\ni z-Ty=z-\rho(Tx)\in C0$ which implies $Y\ni z-Ty\in C^{\flat}$ by \Cref{lem-auto-of-sharp-over-flat}, and thus $ z=Ty$.

$(2\Rightarrow 1)$ 
The pair $(Y\mid \idfunc{Y})$ is a reduction of $(M,C)$ that meets in the radical. 
To see this, note that
\[
\begin{array}{cc}
C^{\sharp}=C^{\flat}+\im(\idfunc{Y}),
     & 
     \{y\in Y\colon \idfunc{Y}(y)\in C^{\flat}\}=Y\cap C^{\flat}=0=\rad(Y),
\end{array}
\]
again since $R$ is semisimple. 
Note also that $y\in Y$ implies $Ty\in Cy$. 
\end{proof}

For use in the sequel we recall a result due to P\v{r}\'{\i}hoda  \cite{Pri2007}. 
For a ring $A$ and an $A$-module $M$ we write $\rad(M)$ for the \emph{radical} of $M$ (the intersection of the maximal, or sum of the superfluous, submodules).





\begin{thm}
\label{thm-isos-have-lifts-that-are-isos}
    \cite[Theorem~2.3]{Pri2007} Let $A$ be a ring, $L$ and $M$ be projective $A$-modules and $h\colon L/\rad(L)\to M/\rad(M)$ be an isomorphism. 
    Then there is an isomorphism $g\colon L\to M$ such that 
    \[
    \begin{array}{cc}
    g(\ell)+\rad(L)=h(\ell+\rad(L)),
         & 
         (\ell\in L).
    \end{array}
    \]
    \end{thm}

For the proof of \Cref{thm-existence-of-reductions}, in \Cref{lem-ideal-annihilating-quotient} we observe how one of the hypothesis yields an isomorphism consistent with the conclusion of 
\Cref{lem-reductions-define-surjective-automorphism-respecting-maps}.

\begin{lem}
    \label{lem-ideal-annihilating-quotient}
    Let $(M,C)$ be an object in $\lRel{R}$. 
    If $IM\subseteq (C^{-1})'+C'$ for some ideal $I$ in $R$ then $C^{\sharp}\cap IM\subseteq C^{\flat}$ and, in particular, $C^{\sharp}/C^{\flat}$ is an $R/I$-module. 
    Furthermore, if $I=\rad(R)$ and $R$ is perfect then there is an $R[T,T^{-1}]$-module $P$ which is projective over $R$ and such that $P/IP\cong C^{\sharp}/C^{\flat}$ as $(R/I)[T,T^{-1}]$-modules. 
\end{lem}

\begin{proof}
    We claim that $IM\cap C^{\sharp}\subseteq C^{\flat}$. 
    Let $z\in IM\cap C^{\sharp}$ and so there exist $x\in(C^{-1})'$ and $y\in C'$  such that $z=x+y$. 
    In particular $x=z-y\in C''$ and similarly $y\in (C^{-1})''$ which altogether shows that $z\in C^{\flat}$ giving the claim. 
In particular $IC^{\sharp}\subseteq C^{\flat}$ and so the $R$-module $C^{\sharp}/C^{\flat}$ is
 annihilated by $I$.  

Now assume $I=\rad(R)$. 
Since $R$ is perfect it is semilocal, and so  $C^{\sharp}/C^{\flat}$ is a module over the semisimple ring $S=R/I$.
A well-known theorem of Bass \cite{Bas1960} says that any module over a perfect ring has a projective cover. 
In particular there exists a projective $R$-module $P$ and an $R$-module isomorphism $\alpha\colon P/IP\to C^{\sharp}/C^{\flat}$; see, for example, the proof of \cite[Proposition~24.12]{Lam1991}. 
By \cite[Proposition~24.6(3)]{Lam1991} we have $\rad(P)=IP$. 
By \Cref{lem-auto-of-sharp-over-flat} there is an $R$-module automorphism $\theta$ of $C^{\sharp}/C^{\flat}$ which is of course $S$-linear.

Let $\tau\coloneqq \alpha^{-1}\theta\alpha$. 
By construction we have that $P/\rad(P)$ is an $S[T,T^{-1}]$-module and that $\alpha$ is an $S[T,T^{-1}]$-module isomorphism since $\alpha\tau=\theta\alpha$. 
Using that $\tau$ is an automorphism of $P/\rad(P)$, and that $P$ is a projective $R$-module, altogether by \Cref{thm-isos-have-lifts-that-are-isos} there is an automorphism $\sigma$ of $P$ with  $\sigma(p)+\rad(P)=\tau(p+\rad(P))$ for all $p\in P$.
Hence $P$ is an $R[T,T^{-1}]$-module by letting $T$ act as $\tau$. 
\end{proof}

\begin{thm}
\label{thm-existence-of-reductions}
    Let $R$ be a local ring with maximal ideal $\jacob$ with division ring denoted by $K=R/J$. 
    If $(M,C)$ is an object of $\lRel{R}$ such that $\jacob M\subseteq (C^{-1})'+C'$ and such that $C^{\sharp}/C^{\flat}$ has finite rank over $K$, then there exists a free reduction $(X\mid \rho)$ of $(M,C)$ that meets in the radical. 
\end{thm} 

\begin{proof}
By \Cref{lem-ideal-annihilating-quotient} that there is an $R[T,T^{-1}]$-module $X$ which is projective over $R$ and a $K[T,T^{-1}]$-module isomorphism $\alpha\colon X/\jacob X\to C^{\sharp}/C^{\flat}$. 
Since $R$ is local $X$ is free over $R$ by the well-known theorem of Kaplansky \cite{Kap1958}. 
Fix an $R$-basis $(b_{\omega}\mid \omega\in \Omega)$ of $X$. 
For each $\omega\in\Omega$ choose $m_{\omega}\in C^{\sharp}$ such that $\alpha(b_{\omega}+JX)=m_{\omega}+C^{\flat}$ and choose $r_{\omega\lambda}\in R$  such that $T(b_{\omega})=\sum_{\lambda\in\Omega}r_{\omega\lambda}b_{\lambda}$. 
It is straightforward to check that the elements $m_{\omega}+C^{\flat}$ with $\omega\in \Omega$ define a $K$-basis of $C^{\sharp}/C^{\flat}$ which means $\Omega$ is a finite set. 
%
%
%
Also,
\[
\begin{array}{c}
\sum_{\lambda\in\Omega}r_{\omega\lambda}m_{\lambda}+C^{\flat}
=
\sum_{\lambda\in\Omega}r_{\omega\lambda}\alpha(b_{\lambda}+\rad(X))
=
\alpha(\sum_{\lambda\in\Omega}r_{\omega\lambda}b_{\lambda}+\rad(X))
\\
=
\alpha(
T(b_{\omega})+\rad(X)
)
=
\theta(
\alpha(b_{\omega}+\rad(X))
)
=\theta(m_{\omega}+C^{\flat})
\end{array}
\]
 for each $\omega\in \Omega$ where $\theta $ is the automorphism from \Cref{lem-auto-of-sharp-over-flat}.
So far note that $X$ has been given the structure of an $R[T,T^{-1}]$-module. 
For each $\lambda,\omega\in \Omega$ 
 write $\delta_{\lambda\omega}$ for the Kronecker delta function, meaning that  we
let $\delta_{\lambda\omega}=1$ if $\lambda=\omega$, and let  $\delta_{\lambda\omega}=0$ if $\lambda\neq \omega$. 
For each $\lambda\in \Omega$ define the elements $q_{\lambda\gamma}\in R$ ($\gamma\in\Omega$) such that $T^{-1}(b_{\lambda})=\sum_{\gamma\in\Omega}q_{\lambda\gamma}b_{\gamma}$. 
By construction we have $\sum_{\lambda\in\Omega}r_{\omega\lambda}q_{\lambda\gamma}=\delta_{\omega\gamma}=\sum_{\lambda\in\Omega}q_{\omega\lambda}r_{\lambda\gamma}$. 

For each $\omega,\lambda\in \Omega$ and each $n\in \mathbb{Z}$ we now define elements $r_{\omega\lambda}^{n\rightarrow},r_{\omega\lambda}^{\leftarrow n}\in R$ as follows. 
Let $r_{\omega\lambda}^{0\rightarrow}\coloneqq q_{\omega\lambda}$, $r_{\omega\lambda}^{\leftarrow 1}\coloneqq -\delta_{\omega\lambda}$, $r_{\omega\lambda}^{n\rightarrow}\coloneqq 0$ for $n>0$ and  $r_{\omega\lambda}^{\leftarrow n}\coloneqq 0$ for $n\leq 0$. 
If $n\leq 0$ and $r_{\gamma\lambda}^{n\rightarrow}$ is defined for each $\gamma\in \Omega$, let $r_{\omega\lambda}^{n-1\rightarrow}\coloneqq\sum_{\gamma\in\Omega}q_{\omega\gamma}r_{\gamma\lambda}^{n\rightarrow}$. 
If $n> 0$ and $r_{\omega\gamma}^{\leftarrow n}$ is defined for each $\gamma\in \Omega$, let 
$r_{\omega\lambda}^{\leftarrow n+1}\coloneqq\sum_{\gamma\in\Omega}r_{\omega\gamma}r_{\gamma\lambda}^{\leftarrow n}$. 

Recall, from \Cref{lem-auto-of-sharp-over-flat}, that for each $\omega\in \Omega$, one has that 
\[
\begin{array}{ccc}
T(b_{\omega})=\sum_{\lambda\in\Omega}r_{\omega\lambda}b_{\lambda}
&
\text{if and only if}
&
\sum_{\lambda\in\Omega}r_{\omega\lambda}m_{\lambda}\in C^{\sharp}\cap (C^{\flat}+C m_{\omega}).
\end{array}
\]
Hence for each $\omega\in \Omega$ there are elements $\ell_{\omega}\in C^{\flat}$ and $p_{\omega}\in Cm_{\omega}$ such that $p_{\omega}=\ell_{\omega}+\sum_{\lambda\in\Omega}r_{\omega\lambda}m_{\lambda}$. 
Fix  $\omega\in \Omega$. 
Since $\ell_{\omega}\in C^{\flat}$ there exists $\ell_{\omega}^{\rightarrow}\in C'\cap (C^{-1})''$ and $\ell_{\omega}^{\leftarrow}\in C''\cap(C^{-1})'$ such that $\ell_{\omega}=\ell_{\omega}^{\leftarrow}+\ell_{\omega}^{\rightarrow}$. 

The notation from \Cref{rem-unpacking-relation-notn} gives $d_{\rightarrow}(\omega),d_{\leftarrow}(\omega)\in\mathbb{Z}$ and $\ell_{\omega}^{n \rightarrow},\ell_{\omega}^{\leftarrow n}\in M$ for each $n\in\mathbb{Z}$ such that
\[
\begin{array}{ccc}
\ell_{\omega}^{0\rightarrow}=\ell_{\omega}^{\rightarrow},
&
\ell_{\omega}^{n \rightarrow}\in C\ell_{\omega}^{n+1 \rightarrow}\text{ for all }n\in \mathbb{Z},
&
\ell_{\omega}^{n \rightarrow}=0\text{ for all }n<d_{\rightarrow}(\omega), 
\\
\ell_{\omega}^{\leftarrow 0}=\ell_{\omega}^{\leftarrow},
&
\ell_{\omega}^{n \leftarrow}\in C\ell_{\omega}^{n+1 \leftarrow}\text{ for all }n\in \mathbb{Z},
&
\ell_{\omega}^{n \leftarrow}=0\text{ for all }n >d_{\leftarrow}(\omega).
\end{array}
\]
For each $n\in \mathbb{Z}$ and each $\omega\in \Omega$ define 
$h_{\omega}^{\leftarrow n}\coloneqq\sum_{\lambda\in\Omega} r_{\omega\lambda}^{\leftarrow n}\ell_{\lambda}^{\leftarrow n}$  and  $h_{\omega}^{n\rightarrow}\coloneqq\sum_{\lambda\in\Omega}  r_{\omega\lambda}^{n\rightarrow }\ell_{\lambda}^{n \rightarrow}$. 

If $n<0$ then $n+1\leq 0$ and hence $r_{\omega\lambda}^{n \rightarrow }=r_{\omega\lambda}^{n+1-1\rightarrow }=\sum_{\zeta\in\Omega}q_{\omega\zeta}r_{\zeta\lambda}^{n+1\rightarrow}$ and therefore
\begin{equation}
\label{eqn-main-recut-theorem-1}
\begin{array}{c}
\sum_{\omega\in \Omega}r_{\gamma\omega}h_{\omega}^{n\rightarrow}
=
\sum_{\omega\in \Omega}r_{\gamma\omega}\sum_{\lambda\in\Omega}  r_{\omega\lambda}^{n\rightarrow }\ell_{\lambda}^{n \rightarrow}
=
\sum_{\omega\in \Omega}r_{\gamma\omega}\sum_{\lambda\in\Omega}  \sum_{\zeta\in\Omega}q_{\omega\zeta}r_{\zeta\lambda}^{n+1\rightarrow} \ell_{\lambda}^{n \rightarrow}
\\
=
\sum_{\lambda\in\Omega}
\sum_{\zeta\in\Omega}
(\sum_{\omega\in \Omega}r_{\gamma\omega}  q_{\omega\zeta})r_{\zeta\lambda}^{n+1\rightarrow} \ell_{\lambda}^{n \rightarrow}
=
\sum_{\lambda\in\Omega}
\sum_{\zeta\in\Omega}
\delta_{\gamma\zeta}r_{\zeta\lambda}^{n+1\rightarrow} \ell_{\lambda}^{n \rightarrow}
=
\sum_{\lambda\in\Omega}
r_{\gamma\lambda}^{n+1\rightarrow} \ell_{\lambda}^{n \rightarrow}. 
\end{array}
\end{equation}
If $n=0$ then $r_{\omega\lambda}^{0\rightarrow}= q_{\omega\lambda}$ and since we have $\sum_{\omega\in\Omega}r_{\gamma\omega}q_{\omega\lambda}=\delta_{\gamma\lambda}$ this gives
\begin{equation}
\label{eqn-main-recut-theorem-2}
\begin{array}{cc}
\sum_{\omega\in\Omega}r_{\gamma\omega}h_{\omega}^{0\rightarrow}
=
\sum_{\omega\in\Omega}r_{\gamma\omega}\sum_{\lambda\in\Omega}  r_{\omega\lambda}^{0\rightarrow }\ell_{\lambda}^{0 \rightarrow}
=
\sum_{\lambda\in\Omega}  (\sum_{\omega\in\Omega}r_{\gamma\omega}q_{\omega\lambda})\ell_{\lambda}^{0 \rightarrow}
=
\sum_{\lambda\in\Omega}  \delta_{\gamma\lambda}\ell_{\lambda}^{0 \rightarrow}
= 
\ell_{\gamma}^{0 \rightarrow}. 
\end{array}
\end{equation}
If $n>0$  then $n+1>0$ and hence $r_{\gamma\lambda}^{\leftarrow n+1}=\sum_{\omega\in \Omega}r_{\gamma\omega}r_{\omega\lambda}^{ \leftarrow n}$ and therefore
\begin{equation}
\label{eqn-main-recut-theorem-3}
\begin{array}{cc}
\sum_{\omega\in \Omega}r_{\gamma\omega}h_{\omega}^{\leftarrow n}
=
\sum_{\omega\in \Omega}r_{\gamma\omega}\sum_{\lambda\in\Omega} r_{\omega\lambda}^{\leftarrow n}\ell_{\lambda}^{\leftarrow n}
=
\sum_{\lambda\in\Omega}
(\sum_{\omega\in \Omega}r_{\gamma\omega} r_{\omega\lambda}^{\leftarrow n})\ell_{\lambda}^{\leftarrow n}
=
\sum_{\lambda\in\Omega}
r_{\gamma\lambda}^{\leftarrow n+1}\ell_{\lambda}^{\leftarrow n}. 
\end{array}
\end{equation}

For each $n\in \mathbb{Z}$ and each $\gamma\in \Omega$ define 
$k_{\gamma}^{\leftarrow n}\coloneqq\sum_{\lambda\in\Omega} r_{\gamma\lambda}^{\leftarrow n+1}\ell_{\lambda}^{\leftarrow n}$  and  $k_{\gamma}^{n\rightarrow}\coloneqq\sum_{\lambda\in\Omega}  r_{\gamma\lambda}^{n+1\rightarrow }\ell_{\lambda}^{n \rightarrow}$. 
So, 
\[
\begin{array}{c}
    (h_{\omega}^{\leftarrow n},k_{\omega}^{\leftarrow n-1})=(\sum_{\lambda\in\Omega} r_{\omega\lambda}^{\leftarrow n}\ell_{\lambda}^{\leftarrow n},\sum_{\lambda\in\Omega} r_{\omega\lambda}^{\leftarrow n}\ell_{\lambda}^{\leftarrow n-1})=\sum_{\lambda\in\Omega} r_{\omega\lambda}^{\leftarrow n}(\ell_{\lambda}^{\leftarrow n},\ell_{\lambda}^{\leftarrow n-1})\in C.
\end{array}
\]
Similarly
 $(h_{\omega}^{n\rightarrow },k_{\omega}^{n-1\rightarrow })\in C$. 
Let $h_{\omega}^{n}=h_{\omega}^{n\leftarrow}+h_{\omega}^{n\rightarrow}$ and  $k_{\omega}^{n}=k_{\omega}^{n\leftarrow}+k_{\omega}^{n\rightarrow}$ for all $n\in\mathbb{Z}$ and $\omega\in\Omega$. 
So
\[
\begin{array}{c}
k_{\gamma}^{0}-\sum_{\omega\in \Omega}r_{\gamma\omega}h_{\omega}^{0}
=k_{\gamma}^{\leftarrow 0}-\sum_{\omega\in \Omega}r_{\gamma\omega}h_{\omega}^{0\rightarrow}
=-\ell_{\gamma}^{\leftarrow 0}-\ell_{\gamma}^{0\rightarrow }=-\ell_{\gamma}.
\end{array}
\]
by \eqref{eqn-main-recut-theorem-2}.
When $n\neq 0$ we have $ k_{\gamma}^{n}=\sum_{\omega\in \Omega}r_{\gamma\omega}h_{\omega}^{n}$ by combining  $\eqref{eqn-main-recut-theorem-1}$ and  \eqref{eqn-main-recut-theorem-3}. 
Using that $\Omega$ is finite, we can define $d\coloneqq 1+\max \{\vert d_{\rightarrow}(\omega)\vert, \vert d_{\leftarrow}(\omega)\vert : \omega\in \Omega\}$, and it follows that $h_{\omega}^{n}=k_{\omega}^{n}=0$ whenever $\vert n\vert >d$, by construction. 
Now let $z_{\omega}\coloneqq m_{\omega}+\sum_{n\in\mathbb{Z}}h_{\omega}^{n}$ for each $\omega\in \Omega$. 
It follows that 
\[
\begin{array}{c}
    \sum_{\omega\in \Omega}r_{\gamma\omega}z_{\omega}
    =
    \sum_{\omega\in \Omega}r_{\gamma\omega}m_{\omega}+\sum_{\omega\in \Omega}\sum_{n\in\mathbb{Z}}r_{\gamma\omega}h_{\omega}^{n}
    =
\sum_{\omega\in \Omega}r_{\gamma\omega}m_{\omega}
    +
    \ell_{\gamma}+\sum_{n\in\mathbb{Z}}k_{\gamma}^{n}
    \\=p_{\gamma}+\sum_{n\in\mathbb{Z}}k_{\gamma}^{n}\in C(m_{\gamma}+\sum_{n\in\mathbb{Z}}h_{\gamma}^{n})=Cz_{\gamma}.
\end{array}
\]
Define the required morphism $\rho\colon X\to M$ by extending the assignment $b_{\omega}\mapsto z_{\omega}$ linearly over $R$
%
By construction we have that $\rho(Tx)\in C\rho (x)$ for any $x\in X$. 

Since the elements  $(m_{\omega}+C^{\flat})$ span the $S$-vector space $C^{\sharp}/C^{\flat}$ it follows that $C^{\sharp}=C^{\flat}+\sum_{\omega\in\Omega}Rm_{\omega}$ and so $C^{\sharp}\subseteq C^{\flat}+\im(\rho)$ and $\im(\rho)\subseteq C^{\sharp}$ giving $C^{\sharp}= C^{\flat}+\im(\rho)$. 
If  $x\in \jacob X$ then  $\rho(x)\in C^{\sharp}\cap \jacob M$ which lies in $C^{\flat}$ by \Cref{lem-ideal-annihilating-quotient}. 
Conversley assume $x\in X$ satisfies $\rho(x)\in C^{\flat}$ and write  $x=\sum_{\omega\in\Omega}r_{\omega}b_{\omega}$ for some $r_{\omega}\in R$. 
Since the elements  $m_{\omega}+C^{\flat}=z_{\omega}+C^{\flat}$ with $\omega\in\Omega$ are linearly independent over $K$, the expression
\[
\begin{array}{c}
0=\rho(x)+C^{\flat}=\rho(\sum_{\omega\in\Omega}r_{\omega}b_{\omega})+C^{\flat}=\sum_{\omega\in\Omega}r_{\omega}z_{\omega}+C^{\flat}=\sum_{\omega\in\Omega}(r_{\omega}+\jacob )(z_{\omega}+C^{\flat})
\end{array}
\]
gives $r_{\omega}\in \jacob $ for all $\omega\in\Omega$, and hence $x\in \jacob X$, as required. 
\end{proof}

By \Cref{prop-reductions-generalise-relations-that-split} it follows that, in case $R$ is a field, to say that there exists a reduction of $(M,C)$ that meets in the radical is equivalent to saying that $C$ is \emph{split} in the sense of \cite[p.~9]{Cra2018}. 
By \Cref{example-thesis-relation-ii} below it follows that \Cref{thm-existence-of-reductions} is a strict generalisation of \cite[Lemma~4.6]{Cra2018}.

\begin{example}
\label{example-thesis-relation-ii}
We make \Cref{example-thesis-relation} more concrete. 
Here we were considering an $R$-linear relation of the form $D=C^{-1}C$ where $C$ is the $R$-linear relation given by a graph. 
It follows that $D=D^{-1}$ and so $D^{\sharp}=D''$ and $D^{\flat}=D'$ by \Cref{notation-for-relations}. 
Let $R$ be a discrete valuation ring with uniformizer $\pi$, and so  $\jacob=R\pi$. 

As before $M$ is the quotient $F/K$ of the free module $F=Rz_{0}\oplus Rz_{1}$ by the submodule $K=R\pi z_{1}$, and we write $\overline{z}_{i}=z_{i} + K$. 
Hence $M\cong R\oplus R/\jacob $ as $R$-modules. 
Recall the $R$-linear relation on $M$ given by 
\[
D=\{
(r\overline{z}_{0} + s\overline{z}_{1},r'\overline{z}_{0} + s'\overline{z}_{1})\colon r,r',s,s'\in R\text{ and } r-r'\in \jacob 
\}.
\]
It follows that $D''=M$ and that $r\overline{z}_{0} + s\overline{z}_{1}\in D'$ if and only if $\pi\mid r$, meaning $D'=\jacob \overline{z}_{0}\oplus (R/\jacob )\overline{z}_{1}$. 
Note that $\jacob M=\jacob \overline{z}_{0}\subseteq D'=D'+(D^{-1})'$ and $D''/D'\cong R/\jacob $ has finite dimension over $K=R/\jacob $. 
Thus by \Cref{thm-existence-of-reductions} there is a reduction $(X\mid \rho)$ of $(M,C)$ that meets in $\jacob $. 
We claim that \Cref{prop-reductions-generalise-relations-that-split} fails in case $R$ is artinian but not a field. 
For then $\jacob^{n}=0$ for some  minimal $n>1$, and it follows that 
\[
\begin{array}{cc}
 D^{\sharp}=D''=M\cong R\oplus R/\jacob,
 &
 D^{\flat}\oplus D^{\sharp}/D^{\flat} = D'\oplus D^{\sharp}/D^{\flat}  \cong \jacob \oplus R/\jacob \oplus R/\jacob      
\end{array}
\]
which gives $ D^{\sharp}\ncong D^{\flat}\oplus D''/D'$ 
since the right-hand side is annihilated by $\jacob^{n-1}$ but the left-hand side is not. 
\end{example}

\section{Linear compactness and covering}
\label{sec-linear-compactness
}


Recall that a \emph{topological} ring (respectively, module) is one with a topology for which the binary operations of multiplication (respectively, the action) and addition, together with the unary operation of taking additive inverses, are continuous. 
For example, a \emph{filtered} ring $R$ comes equipped with a filtration, namely a descending chain of ideals $R_{n}$ ($n\in\mathbb{N}$) such that $R_{p}R_{q}\subseteq R_{p+q}$. 
Likewise, a \emph{filtered} module $M$ over such a ring comes with a descending chain of submodules $M_{n}$ ($n\in\mathbb{N}$) such that $R_{p}M_{q}\subseteq M_{p+q}$, and filtered modules are topological. 
These are the only examples of topological rings and modules that we consider.

\begin{setup}
    \label{setup-filtered}
In \S\ref{sec-linear-compactness
} we will assume that $R$ is a commutative ring which is $\mathbb{N}$-\emph{filtered} by  $\{R_{n}\mid n\geq 0\}$. 
Filtered $R$-modules $M$ will come equipped with a filtration of submodules  denoted $\{M_{n}\mid n\geq 0\}$.  
\end{setup}

In \Cref{rem-basic-comm-alg-properties} we collect together some facts about the topology for a filtered module.


\begin{rem}
    \label{rem-basic-comm-alg-properties} 
    Let $L$ and $M$ be filtered $R$-modules. 
    See \cite[Lemma~8.2.1,
    ~Exercise~8.4]{Singh-balwant-basic-comm-alg}. 
    \begin{enumerate}
        \item \label{rem-fundamental-system-neighbourhood}
The sets $m+M_{n}$ with $m\in M$ and $n\geq 0$ are all clopen sets for the topology on $M$, and together define a fundamental system of neighbourhoods of $x$. 
        In particular, $\{M_{n}\mid n\geq 0\}$ defines a fundamental system of neighborhoods of $0$. 
        This is a nice feature of filtered modules. 
        \item \label{rem-hausdorff-item}The topology on $M$ is Hausdorff if and only if $\bigcap_{n\geq0} M_{n}=0$. 
        A sequence $(m_{n}\mid n\geq 0)$ of elements $m_{n}\in M$ is Cauchy if and only if, for each $n\geq 0$, there exists $d(n)\geq 0$ such that $m_{d+1}-m_{d}\in M_{d(n)}$ for all $d\geq n$. 
        If $M$ is Hausdorff, and if every Cauchy sequence converges, then the universal morphism from $M$ to the inverse limit of $M/M_{n}$ is an isomorphism. 
        \item \label{rem-sub-quo-top-item} Any submodule $N$ of a filtered module $M$ itself is filtered by the subspace topology, given by the filtration $\{N\cap M_{n}\mid n\geq 0\}$. 
        Similarly $M/N$ is filtered by $\{(N+M_{n})/N\mid n\geq 0\}$. 
        \item An $R$-linear homomorphism $f\colon L\to M$ is continuous if and only if, for each $n\geq 0$, there exists $d(n)\geq 0$ such that $f(\ell)\in M_{n}$ for all $\ell\in L_{d(n)}$. 
        \item For a subset $S\subseteq M$ the closure of $S$ is $\bigcap_{n\geq 0}(S+M_{n})$.
        In particular, if $S=m+N$ for a submodule $N$ with $M_{n}\subseteq N$ for $n\gg 0$, it follows that $m+N$ is closed. 
    \end{enumerate}
    By (5), if $m\in M$ then any element $\ell $ lying in the closure $\bigcap_{n\geq 0}(m+M_{n})$ of $\{m\}$ satisfies $\ell-m\in \bigcap_{n\geq0} M_{n}$. 
    Considering (1), this is consistent with the fact that singletons are closed in Hausdorff topological spaces. 
    \begin{enumerate}\setcounter{enumi}{5}
        \item  \label{rem-adic-cont-item} 
For an ideal $I$ of $R$ we will say $R$ (respectively, $M$) has the $I$-\emph{adic} filtration if $R_{n}=I^{n}$ (respectively, $M_{n}=I^{n}M$) for each $n\in \mathbb{N}$. 
    If $R$, $L$ and $M$ all have the $I$-adic filtration then any $R$-linear homomorphism  $ L\to M$ is continuous by taking $d(n)=n$ in (4). 
    \item 
 \label{rem-krull-item} 
Let $\jacob$ be the jacobson radical of $R$. 
 If $R$ is noetherian then $\bigcap_{n\geq 0}J^{n}M=0$ for any finitely generated $R$-module $M$. 
    This is a consequence of the Artin--Rees theorem; see for example \cite[Corollary~8.1.4]{Singh-balwant-basic-comm-alg}. 
    \end{enumerate}
\end{rem}


We will follow the article of Zelinsky \cite{Zel1953}.

\begin{defn}
\label{defn-linea-compactness}
\cite[\S1,~p.~80,~Definition]{Zel1953} 
By a \emph{linear variety} in a filtered module $M$ we mean a coset of a submodule, that is, a subset of the form $m+N$ where $N$ is a submodule of $M$. 

A collection of subsets  of $M$ is said to have the \emph{finite intersection property} provided the intersection of any finite collection of these subsets is non-empty. 
We then say that $M$ is  \emph{linearly
compact} if  any collection of closed linear varieties with the finite intersection property has a non-empty intersection. 
\end{defn}

We will be interested in linear compact modules, and hence closed linear varieties. 

\begin{rem}
\label{rem-finite-intersections-of-closed-linear-varieties}
Recall that if $L$ and $N$ are submodules of $M$ and if $m',m''\in M$ then either $(m'+N)\cap(m''+L)=\emptyset$, or this intersection contains an element $m$ in which case $(m'+N)\cap(m''+L)=m+N\cap L$. 
%
Thus, the intersection of finitely many closed linear varieties is closed, and a linear variety in case it is non-empty. 
\end{rem}

Following results from \cite{Zel1953} we begin with some remarks on linear varieties and linearly compact modules. 
Recall from \Cref{rem-basic-comm-alg-properties}(\ref{rem-sub-quo-top-item}) that submodules and quotients of filtered modules have induced filtrations.

We fix notation for the sequel. 
Let $f\colon L\to M$  be an $R$-linear map, $m\in M$ and  $S\subseteq L$ and $T\subseteq M$ be subsets. 
%
%
Define subsets $f(S)\subseteq M$, $f^{-1}(T)\subseteq L$ and $f^{-1}(m)\subseteq L$, namely,
\[
\begin{array}{cccc}
f(S)\coloneqq \{f(\ell)\mid \ell\in S\},
&
f^{-1}(T)\coloneqq \{\ell\in L\mid f(\ell)\in T\},
&
f^{-1}(m)\coloneqq f^{-1}(\{m\})=\{\ell\in L\mid f(\ell)=m\}. 
\end{array}
\]
If $S$ (respectively, $T$) is an $R$-submodule then so is $f(S)$ (respectively, $f^{-1}(T))$. 
Note $\ker(f)=f^{-1}(0)$.

\begin{lem}
\label{lem-combining-lin-compact-with-relations-prelim}
The following statements hold a continuous $R$-linear map $f\colon L\to M$ of filtered modules. 
\begin{enumerate}
\item Let $m\in M$. 
If $\bigcap_{n\geq 0} M_{n}=0$ then $f^{-1}(m)$ is closed. 
If $f^{-1}(m)\neq \emptyset$ then  $f^{-1}(m)$ is a linear variety.  
\item If $L$ and $M$ are linearly compact then $f(N)$ is closed in $M$ for any closed submodule $N$ of $L$. 
\end{enumerate} 
\end{lem}

\begin{proof}
(1) By \Cref{rem-basic-comm-alg-properties}(\ref{rem-hausdorff-item}) the topology on $M$ is Hausdorff, and so singletons are closed. 
Hence, by assuming $f$ is continuous, the preimage $f^{-1}(m)$ of the closed set $\{m\}$ is closed. 
Now suppose $f^{-1}(m)\neq\emptyset$ and so there is some $\ell\in f^{-1}(m)$ meaning that $\ell\in L$ and $f(\ell)=m$. 
Let $K=f^{-1}(0)$, the kernel of $f$, which is a submodule of $L$. 
We claim $f^{-1}(m)=\ell+K$. 
Clearly $f^{-1}(m)\supseteq \ell+K$. 
Now for any $\ell'\in f^{-1}(m)$ we have $\ell'=\ell+\ell'-\ell$ where $\ell'-\ell\in K$ $f(\ell')=m=f(\ell)$, giving the reverse inclusion.

(2)  By  \cite[Proposition~3]{Zel1953} we have that $N$ is linearly compact, since it is closed inside the linearly compact module $L$. 
By \cite[Proposition~2]{Zel1953} this means the image $f(N)$ of $N$ under $f$ is again linearly compact. 
Recall from \Cref{rem-basic-comm-alg-properties}(\ref{rem-fundamental-system-neighbourhood})  that the submodules $M_{n}$ define a fundamental system of neighbourhoods of $0\in M$. 
By  \cite[Proposition~7]{Zel1953} we have that $f(N)$ is closed inside $M$ since it is a linearly compact submodule of $M$.  
\end{proof}

We note how the conclusion of \Cref{lem-combining-lin-compact-with-relations-prelim} holds if one swaps the graph of $f$ with the inverse relation.

\begin{rem}
\label{rem-converse-to-relations-prelim-lin-compact}
Let $f\colon L\to M$ be a continuous $R$-module homomorphism. 
For any $\ell \in L$ we have $\{f(\ell)\}=f(\ell)+0$. Clearly this singleton is a non-empty linear variety, which closed in $M$ if $\bigcap_{n\geq 0} M_{n}=0$, since then $M$ is Hausdorff by \Cref{rem-basic-comm-alg-properties}(\ref{rem-hausdorff-item}). 
Also, if $N$ is a closed submodule of $M$ then the preimage $f^{-1}(N)$ is a submodule since $f$ is $R$-linear, and closed since $f$ is continuous. 
\end{rem}

We are interested in diagrams in $\lMod{R}$ of certain shapes, generalising the category defined by the poset $\mathbb{N}$. 
Such shapes are encoded in a quiver whose underlying graph is a binary tree. 

\begin{defn}
\label{defn-A-inf-reps}
   
%
Recall the quiver $\mathcal{B}$ from \eqref{binary-tree}, whose underlying graph is the rooted complete binary tree. 
We give a precise definition of $\mathcal{B}$, as follows. 
   Let $[1]^{[0]}\coloneqq\emptyset$ and for an integer $n>0$ let $[1]^{[n]}$ be the set of sequences $\sigma\colon \{1,\dots,n\}\to  \{0,1\}$. 
 Let  $\mathcal{B}_{0}\coloneqq \bigcup_{n\in\mathbb{N}}[1]^{[n]}$ where $n$ defines the depth of a vertex in the underlying binary tree. 
For $\sigma\in [1]^{[n]}$ define $(\sigma\pm)\in [1]^{[n+1]}$  by $(\sigma\pm)(i)\coloneqq \sigma(i)$ for $i\leq n$, $(\sigma+)(n+1)\coloneqq 1$ and $(\sigma-)(n+1)\coloneqq 0$. 
Let $\mathcal{B}_{1}=\bigcup_{\sigma\in [1]^{[n]}}\{\Rceil_{\sigma},\Rfloor_{\sigma}\}$ and define head and tail functions $h_{\mathcal{B}},t_{\mathcal{B}}\colon \mathcal{B}_{1}\to \mathcal{B}_{0}$ by
\[
\begin{array}{cccc}
    h_{\mathcal{B}}(\Rceil_{\sigma})
    \coloneqq
    \sigma+,
     &
     t_{\mathcal{B}}(\Rceil_{\sigma})
     \coloneqq
     \sigma,
     &
     h_{\mathcal{B}}(\Rfloor_{\sigma})
     \coloneqq
     \sigma,
     &
     t_{\mathcal{B}}(\Rfloor_{\sigma})
     \coloneqq
     \sigma-.
\end{array}
    \]
A \emph{ray}\footnote{This is a term from graph theory, in that a ray here is the same a (rooted) ray for the underlying graph of $\mathcal{B}$; see for example \cite[pp.~202-203,~Figure~8.1.4]{diestel-graph}.} of the binary tree is a sequence $\overline{\sigma}=(\sigma^{i})$ of depth-$i$ vertices where $\sigma^{i+1}=\sigma^{i} \pm$ for each $i$. 
We define a \emph{ray}-\emph{representation} as the restriction of a $\mathcal{B}$-diagram in $\lMod{R}$ to the full subcategory given by a ray.
\end{defn}

\begin{rem}
    Any ray $\overline{\sigma}=(\sigma^{i})$ necessarily starts with the root $\sigma^{0}=\emptyset$, since this is the unique depth-$0$ vertex.  
    The notation $\sigma\pm $ is supposed to be mnemonic, in that  $\sigma+$ (respectively, $\sigma-$) is found by travelling from $\sigma$ upward and forwards along $\Rceil_{\sigma}$ (respectively, downward and backwards along $\Rfloor_{\sigma}$).  
    For example $(((\emptyset-)+)+)-=((0+)+)-=(01+)-=011-=0110$ is found by travelling along $\Rfloor$, $\Rceil_{0}$, $\Rceil_{01}$ and $\Rfloor_{011}$.  
\end{rem}
\begin{notn}
\label{notn-for-rays}
There is a bijection between rays and functions $\mathbb{N}_{>0}\to \{-,+\}$ where  $\overline{\sigma}=(\sigma^{i})$ corresponds to the  function $\underline{\sigma}=(\sigma_{i})\in\prod_{i>0}\{-,+\}$
if and only if $\sigma^{i+1}=\sigma^{i}\sigma_{i+1}$ for each $i\in\mathbb{N}$. 
If $(\sigma^{i})$ is a ray and $\mathscr{M}$ is  $\mathcal{B}$-diagram in $\lMod{R}$ let $(M[i],\mu_{i})$ denote\footnote{Clearly we are suppressing the dependency of $(M[i],\mu_{i})$ on $\overline{\sigma}$. 
This does not cause problems in the sequel. } the corresponding ray-representation. 
That is,  $M[i]\coloneqq \mathscr{M}(\sigma^{i})$, $\mu_{i}\coloneqq \mathscr{M}(\Rfloor_{\sigma^{i}})$ when $\sigma_{i}=-$ and $\mu_{i}\coloneqq \mathscr{M}(\Rceil_{\sigma^{i}})$ when $\sigma_{i}=+$. 
Any ray-representation  $(M[i],\mu_{i})$ is depicted 
\[
\begin{tikzcd}
M[0]
\arrow[r, -, "{\mu_{1}}"]
&
M[1]
\arrow[r, -, "{\mu_{2}}"]
&
M[2]
\arrow[r, -]
&
\cdots
\arrow[r, -]
&
M[n-1]
\arrow[r, -, "{\mu_{n}}"]
&
M[n]
\arrow[r, -]
&
\cdots
\end{tikzcd}
\]
where $\mu_{i}$ points left when $\sigma_{i}=-$ and right when $\sigma_{i}=+$. 
Given $(M[i],\mu_{i})$ we define $R$-linear relations  
%
 \[
 \begin{array}{cc}
 \widetilde{\mu}_{i}\coloneqq\{(m,\mu_{i}(m))\mid m\in M[i]\}\text{ when }\sigma_{i}=-,
 &
 \widetilde{\mu}_{i}\coloneqq\{(\mu_{i}(m'),m')\mid m'\in M[i-1]\}\text{ when }\sigma_{i}=+.
 \end{array}
 \]
Finally let $\mu_{0}=\idfunc{M[0]}$, $\widetilde{\mu}_{\leq 0}=\graph{\idfunc{M[0]}}$ and, for each $n> 0$, let $\widetilde{\mu}_{\leq n}\coloneqq \widetilde{\mu}_{1}\dots \widetilde{\mu}_{n}$ following \Cref{notn-composing-linear-relations}. 
\end{notn}

\begin{example}
\label{example-unpacking-a-inf-diagrams}
We unpack \Cref{defn-A-inf-reps} and \Cref{notn-for-rays}. 
Consider the ray defined by $\sigma_{i}=+$ if and only if $i$ is even. 
So $\overline{\sigma}=(\emptyset,1,10,101,1010,\dots)$. 
For a $\mathcal{B}$-diagram $\mathscr{M}$ in $\lMod{R}$ consider the ray-representation
\[
\begin{tikzcd}[column sep = 1.2cm]
M[0]
\arrow[r, "{\mu_{1}}", "{\mathscr{M}(\Rceil)}"']
&
M[1]
&
M[2]\arrow[l, "{\mathscr{M}(\Rfloor_{1})}", "{\mu_{2}}"']
\arrow[r, "{\mu_{3}}", "{\mathscr{M}(\Rceil_{10})}"']
&
M[3]
&
M[4]\arrow[l, "{\mathscr{M}(\Rfloor_{101})}", "{\mu_{4}}"']
\arrow[r, "{\mu_{5}}", "{\mathscr{M}(\Rceil_{1010})}"']
&
M[5]
&
M[6]\arrow[l, "{\mathscr{M}(\Rfloor_{10101})}", "{\mu_{6}}"']
\arrow[r]
&
\cdots
\end{tikzcd}
\]
Combining with \Cref{notn-generalised-image} we have, for any subset $S\subseteq M[5]$, that
\[
\widetilde{\mu}_{\leq 5}S=\widetilde{\mu}_{1}\widetilde{\mu}_{2}\widetilde{\mu}_{3}\widetilde{\mu}_{4}\widetilde{\mu}_{5}S=\left\{m\in M[0]\left\vert\, \begin{array}{c}
\text{there exists } m'\in M[2],m''\in M[4]\text{ such that }
\\
\mu_{1}(m)=\mu_{2}(m'),\mu_{3}(m')=\mu_{4}(m''),\mu_{5}(m'')\in S
\end{array}\right.\right\}.
\]
\end{example} 

Our focus, for the remainder of the paper, is understanding the $R$-submodule $ \bigcap_{i\in\mathbb{N}}\widetilde{\mu}_{\leq i}M[i]$ of $M[0]$. 

\begin{lem}
\label{lem-finite-covering-cases}
The following statements hold for a ray-representation $(M[i],\mu_{i})$. 
\begin{enumerate}
\item If   $i\in \mathbb{N}$ then  $\widetilde{\mu}_{\leq i}M[i]\supseteq \widetilde{\mu}_{\leq i+1}M[i+1]$ and $\widetilde{\mu}_{\leq i}0\subseteq \widetilde{\mu}_{\leq i+1}0$ and $\bigcup_{i\in \mathbb{N}}\widetilde{\mu}_{\leq i}0\subseteq \bigcap_{i\in\mathbb{N}}\widetilde{\mu}_{\leq i}M[i]$. 
\item If $K= \{0\}\cup \{j\mid \sigma_{j}=-\}$  is finite then $ \bigcap_{i\in\mathbb{N}}\widetilde{\mu}_{\leq i}M[i]= \widetilde{\mu}_{\leq k}M[k]$ where $k=\max K$. 
\item If $K= \{0\}\cup \{j\mid \sigma_{j}=+\}$  is finite then $ \bigcup_{i\in\mathbb{N}}\widetilde{\mu}_{\leq i}0= \widetilde{\mu}_{\leq k}0$ where $k=\max K$. 
\end{enumerate}
\end{lem}

\begin{proof}
By construction, for $i\in\mathbb{N}$,  $m\in\widetilde{\mu}_{\leq i}M[i]$ if and only if we have a tuple of elements  $(m_{0},\dots,m_{i})\in M[0]\times \dots \times M[i]$ such that $m=m_{0}$ and $m_{h-1}\in \widetilde{\mu}_{h}m_{h}$ whenever $0<h\leq i$. 
Likewise, to say that $m\in\widetilde{\mu}_{\leq i}0$ is to say that such a tuple exists where $m_{i}=0$.  
We use this notation in what follows. 

(1) The first inclusion is trivial. 
The other inclusions follow by observing that $(0,0)\in \widetilde{\mu}_{h}$ for any $h\in\mathbb{N}$. 

(2) By (1) we have $\bigcap_{i\in\mathbb{N}}\widetilde{\mu}_{\leq i}M[i]=  \bigcap_{i\geq k}\widetilde{\mu}_{\leq i}M[i] \subseteq \widetilde{\mu}_{\leq k}M[k]$. 
Now let $m\in \widetilde{\mu}_{\leq k}M[k]$ and construct a tuple $(m_{0},\dots,m_{k})$ as above. 
For $j>k$ we require an extension of this sequence. 
Namely, we require a tuple $(m_{k+1},\dots,m_{j})$ where  $m_{h-1}\in \widetilde{\mu}_{h}m_{h}$ whenever $k<h\leq j$. 
For each such $h$ we have $\sigma_{h}=+$ by construction, so $m_{h-1}\in \widetilde{\mu}_{h}m_{h}$ holds if and only if $m_{h}=\mu_{h}(m_{h-1})$, and so it suffices to let $m_{h}=\mu_{h}(\dots(\mu_{k+1}(m_{k}))\dots)$.   

(3) Clearly $\bigcup_{i\in\mathbb{N}}\widetilde{\mu}_{\leq i}0\supseteq \widetilde{\mu}_{\leq k}0$. 
Let $i\in \mathbb{N}$ and $m\in \widetilde{\mu}_{\leq i}0$. 
We claim $m\in \widetilde{\mu}_{\leq k}0$. 
By (1) it suffices to assume $i>k$. 
As above we have a tuple $(m_{0},\dots,m_{i})$ where $m_{i}=0$. 
Since $\sigma_{k+1}=\dots=\sigma_{i}=-$ by construction, it follows that $m_{k}=\mu_{k+1}(m_{k+1})=\dots=\mu_{k+1}(\dots(\mu_{i}(m_{i}))\dots)=0$ as required. 
\end{proof}

\Cref{weak-covering-lemma} below generalises the \emph{weak covering lemma} \cite[Lemma~10.3]{Cra2018} to the language of relations. 

\begin{lem}
\label{weak-covering-lemma}
Let $\mathscr{M}$ be a $\mathcal{B}$-diagram in $\lMod{R}$ such that $\im (\mathscr{M}(\Rfloor_{\sigma}))\subseteq \ker(\mathscr{M}(\Rceil_{\sigma}))$ for all $\sigma\in\mathcal{B}_{0}$. 
If $0\neq m\in \mathscr{M}(\emptyset)$ then there exists a ray $(\sigma^{i})$ in $\mathcal{B}$
such that exactly one of the following statements holds.
\begin{enumerate}
\item There exists $j\in \mathbb{N}$ such that $ m\in \widetilde{\mu}_{\leq j} \ker(\mathscr{M}(\Rceil_{\sigma^{j}}))\setminus \widetilde{\mu}_{\leq j}  \im (\mathscr{M}(\Rfloor_{\sigma^{j}}))$.
\item For each $j\in \mathbb{N}$ we have $m\in \widetilde{\mu}_{\leq j}M[j]\setminus \widetilde{\mu}_{\leq j}0$, and so $m\in \bigcap_{i\in\mathbb{N}}\widetilde{\mu}_{\leq i}M[i]\setminus \bigcup_{i\in \mathbb{N}}\widetilde{\mu}_{\leq i}0$.  
\end{enumerate}
\end{lem}

\begin{proof}
We begin by explaining why, for any ray, (1) and (2) cannot hold simultaneously.
Suppose that (1) holds for some ray $(\sigma^{i})$ and some $j\in\mathbb{N}$, and choose $m_{j}\in \ker(\mathscr{M}(\Rceil_{\sigma^{j}}))\subseteq  M[j]$ such that $m\in \widetilde{\mu}_{\leq j}m_{j}$. 
If $\sigma_{j}=-$ then $\mu_{j+1}=\mathscr{M}(\Rfloor_{\sigma^{j}})$ which means 
$m_{j}=\mu_{j+1}(m_{j+1})$ for some $m_{j+1}\in M[j+1]$, but this is impossible, since then $m\in \widetilde{\mu}_{\leq j}\im(\mathscr{M}(\Rfloor_{\sigma^{j}}))$. 
If $\sigma_{j}=+$ then $\mu_{j+1}=\mathscr{M}(\Rceil_{\sigma^{j}})$ which gives $\mu_{j+1}(m_{j})=0$ and so $m\in\widetilde{\mu}_{\leq j+1}0$, meaning that (2) does not hold. 
So (1) and (2) are mutually exclusive. 

We now assume (1) is false for all rays, and prove (2) holds for some ray $(\sigma^{i})$. 
We begin by iteratively defining $(\sigma^{i})$. 
Suppose that $j\geq 0$ and that  we have chosen 
a sequence $(\sigma^{0},\dots,\sigma^{j})$ of vertices in $\mathcal{B}$, such that if $0\leq i<j$ then $\sigma^{i+1}=\sigma^{i}\pm$. 
Let $M[0],\dots,M[j]$ be the images of $(\sigma^{0},\dots,\sigma^{j})$ under $\mathscr{M}$. 
Let $\tau=\sigma^{j}$.
Since we are assuming (1) to be false, it is impossible that both $m\in \widetilde{\mu}_{\leq j}\ker(\mathscr{M}(\Rceil_{\tau}))$ and $m\notin \widetilde{\mu}_{\leq j}\im (\mathscr{M}(\Rfloor_{\tau}))$.

It follows that we must have $m\notin\widetilde{\mu}_{\leq j}\ker(\mathscr{M}(\Rceil_{\tau}))$ or $m\in \widetilde{\mu}_{\leq j}\im (\mathscr{M}(\Rfloor_{\tau}))$, and not both, since $\im (\mathscr{M}(\Rfloor_{\tau}))$ is contained in $\ker(\mathscr{M}(\Rceil_{\tau}))$ by hypothesis.  
If $m\notin\widetilde{\mu}_{\leq j}\ker(\mathscr{M}(\Rceil_{\tau}))$ then let $\sigma_{j+1}=+$, and otherwise if 
$m\in \widetilde{\mu}_{\leq j}\im (\mathscr{M}(\Rfloor_{\tau}))$ then let  $\sigma_{j+1}=-$. 

With $(\sigma^{i})$ defined, it suffices to let $j\geq0$ and prove  $m\in  \widetilde{\mu}_{\leq j}M[j]\setminus \widetilde{\mu}_{\leq j}0$. 
Since $m\neq 0$ and $\widetilde{\mu}_{\leq 0}=\graph{\idfunc{M[0]}}$, the case where $j=0$ is immediate, and we proceed by induction. 
So suppose $m\in  \widetilde{\mu}_{\leq j}m'$ for some $m'\neq 0$.  
On the one hand, when $\sigma_{j+1}=+$ we have $\mu_{j+1}=\mathscr{M}(\Rceil_{\tau})$  by construction, and so setting $m''=\mu_{j+1}(m')$ we must have $m''\neq 0$. 
On the other hand, when  $\sigma_{j+1}=-$ we have $\mu_{j+1}=\mathscr{M}(\Rfloor_{\tau})$ by construction, and since $0\neq m'\in \im(\mu_{j+1})$ there must exist $m''\in M[j+1]$ with $m''\neq 0$ and $m'=\mu_{j+1}(m'')$. 
In either case it follows that $m\in  \widetilde{\mu}_{\leq j+1}M[j+1]\setminus \widetilde{\mu}_{\leq j+1}0$, as required. 
\end{proof}

In \Cref{prop-prod-and-coprod-work-nice} we note how ray-representations behave well with respect to  products and coproducts. 

\begin{prop}
\label{prop-prod-and-coprod-work-nice}
Let $\overline{\sigma}$ be a fixed ray and  $S$ be an index set. 
For each $s\in S$ let $\mathscr{M}^{s}$ be a $\mathcal{B}$-diagram in $\lMod{R}$ and $(M^{s}[i],\mu^{s}_{i})$ be the corresponding ray-representation. 
    The following statements hold for  $n\in \mathbb{N}$. 
    \begin{enumerate}
        \item The product and coproduct give ray-representations $(M^{\Pi}[i],\mu^{\Pi}_{i}),\,(M^{\oplus}[i],\mu^{\oplus}_{i})$ respectively. 
        \item For the relations $\widetilde{\mu}_{i}^{\Pi}$ associated to $(M^{\Pi}[i],\mu^{\Pi}_{i})$ we have $\widetilde{\mu}^{\Pi}_{\leq n}M^{\Pi}[n]=\prod_{s\in S} \widetilde{\mu}^{s}_{\leq n}M^{s}[n]$. 
        \item For the relations $\widetilde{\mu}^{\oplus}_{i}$ associated to $(M^{\oplus}[i],\mu^{\oplus}_{i})$ we have $\widetilde{\mu}^{\oplus}_{\leq n}M^{\oplus}[n]=\bigoplus_{s\in S} \widetilde{\mu}^{s}_{\leq n}M^{s}[n]$. 
    \end{enumerate}
\end{prop}

\begin{proof}
    (1) 
Let $M^{\Pi}[i]=\prod_{s\in S}M^{s}[i]$, $M^{\oplus}[i]=\bigoplus_{s\in S}M^{s}[i]$, $\mu^{\Pi}_{i}((m^{s}))=(\mu^{s}_{i}(m^{s}))$ and $\mu^{\oplus}_{i}((\ell^{s}))=(\mu^{s}_{i}(\ell^{s}))$ for each $i\in \mathbb{N}$ and $\ell^{s},m^{s}\in M^{s}[i]$ where $\ell^{s}=0$ for all but finitely many $s$. 
    
    Since $\sigma^{i}$ is fixed as $s$ varies, each $(M^{s}[i],\mu^{s}_{i})$ is a diagram in $\lMod{R}$ of the same shape. 
    The claim follows since limits and colimits in the category of functors $\overline{\sigma}\to \lMod{R}$ are computed pointwise. 

    (2) Let $(m^{s}_{0})\in \widetilde{\mu}^{\Pi}_{\leq n}M^{\Pi}[n]$, meaning there are elements $(m^{s}_{i})\in M^{\Pi}[i]$ such that $(m^{s}_{i-1})\in \widetilde{\mu}^{\Pi}_{i}(m^{s}_{i})$ for each $1\leq i \leq n$. 
    For each $i$ and $s$ write $\pi_{i}^{s}\colon M^{\Pi}[i]\to M^{s}[i]$ for the canonical projection. 
    If  $\sigma_{i}=-$ we have 
    \[
    \begin{array}{c}
    m^{s}_{i-1}
    =
    \pi_{i}^{s}((m^{s}_{i-1}))
    =
    \pi_{i}^{s}(\mu^{\Pi}_{i}((m^{s}_{i})))
    =
    \pi_{i}^{s}((\mu^{s}_{i}(m^{s}_{i}))
    =\mu_{i}^{s}(m^{s}_{i}),
    \end{array}
    \]
    and, similarly, if $\sigma_{i}=+$ then $m^{s}_{i}=\mu_{i-1}^{s}(m^{s}_{i-1})$ for any $s$ and $i$ with $1\leq i\leq n$. 
    Thus $m_{0}^{s}\in \widetilde{\mu}_{\leq n}^{s}M^{s}[n]$. 
    
    Conversely, if $m_{0}^{s}\in \widetilde{\mu}_{\leq n}^{s}M^{s}[n]$ for each $s$ one can use a similar argument to prove $(m^{s}_{0})\in \widetilde{\mu}^{\Pi}_{\leq n}M^{\Pi}[n]$. 

    (3) The argument in (2) holds when, for each $i$, we have $m_{i}^{s}=0$ for all but finitely many $s$. 
    \end{proof}
In the sequel we adapt the \emph{realization lemma}  \cite[Lemma~10.2]{Cra2018} to the language of relations. 
This combines the notion of linear compactness recalled in \Cref{defn-linea-compactness} with the following concept. 

\begin{defn}
    \label{def-compact-cont-haus}
   A ray-representation $(M[i],\mu_{i})$  is \emph{compact}-\emph{continuous}-\emph{Hausdorff} if, for each $i\in \mathbb{N}$,  $M[i]$ is filtered and linearly compact, $\mu_{i}$ is continuous and  $\bigcap_{n\geq0}M_{n}[i]=0$ (so the topology is Hausdorff). 
\end{defn}

In \Cref{lem-compact-cont-haus-ainf-reps-exist} we provide sufficient conditions for \Cref{def-compact-cont-haus} to hold. 

\begin{lem}
\label{lem-compact-cont-haus-ainf-reps-exist}
    Let $(M[i],\mu_{i})$ be an  ray-representation and assume that $R$ and each $M[i]$ have the  $\jacob$-adic filtration where $\jacob$ is the jacobson radical of $R$. 
    If $R$  is $J$-adically complete, noetherian and semilocal\footnote{This means that $S=R/\jacob$ is a semisimple ring. Since $R$ is commutative, this is equivalent to requiring that $R$ has  finitely many maximal ideals, or to requiring that $S$ is a finite product of fields.} and if  each $M[i]$ is finitely generated, then $(M[i],\mu_{i})$  is compact-continuous-Hausdorff. 
\end{lem}

\begin{proof}
Note that here we are assuming that $R_{n}=\jacob^{n}$ and $M_{n}[i]=\jacob^{n}M[i]$ for each $n\geq 0$.  

By \Cref{rem-basic-comm-alg-properties}(\ref{rem-adic-cont-item}) we have that each $\mu_{i}$ is continuous, and by \Cref{rem-basic-comm-alg-properties}(\ref{rem-krull-item}) we have that each $M[i]$ is Hausdorff. 
Thus it suffices to prove finitely generated $R$-modules must be linearly compact. 

Firstly, observe that artinian filtered modules are linearly compact, since by \cite[Proposition~5]{Zel1953}, any filtered module with the descending chain condition on closed submodules must be linearly compact. 
Secondly, observe that the inverse limit of a system of filtered (respectively,  linearly compact) modules is filtered (respectively,  linearly compact) modules; 
see for example \cite[\S8.2.2]{Singh-balwant-basic-comm-alg} (respectively, 
 \cite[Proposition~4]{Zel1953}).

Since $R$ is noetherian and semilocal we have that $R/J^{n}$ has finite-length for each $n>0$; see for example \cite[Exercise~7.16]{Singh-balwant-basic-comm-alg}.  
Furthermore by \Cref{rem-basic-comm-alg-properties}(\ref{rem-sub-quo-top-item}) $R/J^{n}$ is filtered, and hence linearly compact by our first observation. 
Since $R$ is $\jacob$-adically complete, as an $R$-module it is linearly compact by our second observation. 
As noted in \cite[Proposition~1]{Zel1953}, any product of linearly compact modules is again linearly compact, and so $R^{d}$ is a linearly compact $R$-module for any $d>0$. 
By \Cref{rem-basic-comm-alg-properties}(\ref{rem-adic-cont-item})  and  \cite[Proposition~2]{Zel1953} it follows that any finitely generated $R$-module is linearly compact, as required. 
\end{proof}

In \Cref{lem-closure-lin-varieties-with-relations} we consider consequences of \Cref{lem-combining-lin-compact-with-relations-prelim} in terms of \Cref{def-compact-cont-haus}. 
This result may be considered an infinite version of the phenomena discussed in \Cref{example-unpacking-a-inf-diagrams}. 

\begin{lem}
\label{lem-closure-lin-varieties-with-relations}
Let $(M[i],\mu_{i})$ be a compact-continuous-Hausdorff ray-representation and let $i\in \mathbb{N}$. 
\begin{enumerate}
\item If $m\in M[i+1]$ then $\widetilde{\mu}_{i}m$ is closed in $M[i]$ and $\widetilde{\mu}_{i}m$ is a linear variety if it is non-empty.  
\item If $N$ is a closed submodule of $M[i+1]$ then $\widetilde{\mu}_{i}N$ is a closed submodule of $M[i]$. 
\end{enumerate} 
Consequently, we have the following equality of $R$-submodules of $M[0]$
\[
\begin{array}{c}
\bigcap_{i\geq0} \widetilde{\mu}_{\leq i}M[i]=\left\{m\in M[0]\left\vert\, \begin{array}{c}
\text{there exists }(m_{i})\in \prod_{i\geq 0}M[i]\text{ such that }
\\
 m=m_{0}\text{ and }(m_{i},m_{i-1})\in \widetilde{\mu}_{i}\text{ for }i>0
\end{array}\right.\right\}.
\end{array}
\]
\end{lem}

\begin{proof}
For the proof of (1) and (2) one combines \Cref{lem-combining-lin-compact-with-relations-prelim}, \Cref{rem-converse-to-relations-prelim-lin-compact} and \Cref{defn-A-inf-reps}. 
It remains to show the equality holds. 
Given $m$ in the righthand-side with $(m_{i})\in \prod_{i\geq 0}M[i]$ such that $m_{i-1}\in \widetilde{\mu}_{i}m_{i}$ for each $i$, for each $\ell \in\mathbb{N}$ we have the following string of containments proving $m$ lies in the lefthand-side
\[
m=m_{0}\in \widetilde{\mu}_{1}m_{1}\subseteq \widetilde{\mu}_{1}\widetilde{\mu}_{2}m_{2}\subseteq\cdots\subseteq \widetilde{\mu}_{1}\dots \widetilde{\mu}_{\ell -1}m_{\ell -1}\subseteq \widetilde{\mu}_{1}\dots \widetilde{\mu}_{\ell }m_{\ell }\subseteq \widetilde{\mu}_{1}\dots \widetilde{\mu}_{\ell }M[\ell ]=\widetilde{\mu}_{\leq \ell }M[\ell ]
\]
It remains to show the lefthand-side is contained in the righthand-side of the required equality. 
Let 
\[
\begin{array}{cc}
L(j,\ell )\coloneqq \widetilde{\mu}_{j+1}\dots \widetilde{\mu}_{\ell }M[\ell ] \quad (\ell > j),
&
N(j)\coloneqq \widetilde{\mu}_{j}^{-1}m_{j-1}\cap \bigcap_{\ell >j} L(j,\ell )  \quad (j>0).
\end{array}
\]   
Let  $N(0)\coloneqq  \bigcap_{\ell >0} L(0,\ell )$ and note that $L(0,\ell )=\widetilde{\mu}_{\leq \ell }M[\ell ]$.  
For any $j\in \mathbb{N}$ note that $L(j,\ell )$ and $N(j)$ are subsets of $M[j]$ and that $L(j,\ell )$ is an $R$-submodule. 
Starting from setting $m_{0}\coloneqq m\in N(0)$, we shall construct $(m_{i})\in \prod_{i\geq 0}N(i)$  iteratively. 
So let $j>0$ and assume $m_{i}\in N(i)$ has been defined for $i=0,\dots,j-1$. 
We require $m_{j}\in N(j)$. 
It remains to prove that $N(j)\neq \emptyset$. 

By \Cref{lem-closure-lin-varieties-with-relations} we have that $N(j)$ is an intersection of closed linear varieties. 
Since $M[j]$ is linearly compact, it suffices to prove that these linear varieties have the finite intersection property. 
Let $S$ be a finite non-empty subset of integers $\ell >j$ and let $d\coloneqq \max S$. 
The intersection of the linear varieties $L(j-1,\ell )$ with $\ell \in S$ is  $\bigcap_{\ell \in S}L(j-1,\ell )=L(j-1, d)\ni 0$, and so it suffices to prove that $\widetilde{\mu}_{j}^{-1}m_{j-1}\cap L(j,d)\neq \emptyset$.  
Indeed, $m_{j-1}\in \widetilde{\mu}_{j}L(j,d)$ and so there is some $m'\in L(j,d)$ such that $m_{j-1}\in \widetilde{\mu}_{j}m'$  as required.  
\end{proof}

We now come to our adaptation of the covering property  \cite[Lemma~10.2]{Cra2018} to the language of relations.

\begin{thm}
\label{thm-main-3}
    Let $R$ be a $J$-adically complete noetherian semilocal ring.  
    Let $\mathscr{M}$ be a $\mathcal{B}$-diagram in $\lMod{R}$ such that $\mathscr{M}(\sigma)$ is finitely generated and such that   $\im (\mathscr{M}(\Rfloor_{\sigma}))\subseteq \ker(\mathscr{M}(\Rceil_{\sigma}))$ for each $\sigma\in\mathcal{B}_{0}$. 
    If $0\neq m\in \mathscr{M}(\emptyset)$ then there exists a ray $\overline{\sigma}$ in $\mathcal{B}$
such that exactly one of the following statements holds.
\begin{enumerate}
\item There exists $j\in \mathbb{N}$ such that $ m\in \widetilde{\mu}_{\leq j} \ker(\mathscr{M}(\Rceil_{\sigma^{j}}))\setminus \widetilde{\mu}_{\leq j}  \im (\mathscr{M}(\Rfloor_{\sigma^{j}}))$. 
\item There exists $(m_{i})\in \prod _{i\in \mathbb{N}}M[i]$ such that $m_{0}=m$ and $m_{i-1}\in \widetilde{\mu}_{i}m_{i}$ when $i>0$, but $m\notin \bigcup_{i\in \mathbb{N}}\widetilde{\mu}_{\leq i}0$. 
\end{enumerate}
\end{thm}

\begin{proof}
This is the combination of \Cref{weak-covering-lemma}, \Cref{lem-compact-cont-haus-ainf-reps-exist} and \Cref{lem-closure-lin-varieties-with-relations}. 
\end{proof}

\begin{acknowledgements}
   This work was supported by the Deutsche Forschungsgemeinschaft (SFB-TRR 358/1 2023- 491392403). 
   The author is grateful for this support.
   The author is also grateful to Francesca Fedele and Rosie Laking for helpful and encouraging conversations.  
\end{acknowledgements}

\bibliography{biblio}
\bibliographystyle{abbrv}

\end{document}